\documentclass[a4paper,11pt]{amsart}

\usepackage{amssymb, amsthm, amsmath}
\usepackage{tikz}

\usepackage{bbm}

\usepackage[hmargin=3.5cm,foot=0.7cm]{geometry}
\usepackage[utf8x]{inputenc}

\usepackage[english]{babel}
\usepackage[abbrev]{amsrefs}
\usepackage{xcolor}

\usepackage{tikz}
\usepackage{tikz-cd}
\usetikzlibrary{matrix,arrows}
\usepackage{pgfplots}
\pgfplotsset{width=7cm,compat=1.8}
\usepgfplotslibrary{polar}

\usepackage{hyperref}

\usetikzlibrary{decorations.pathreplacing}

\usepackage[colorinlistoftodos,prependcaption,textsize=tiny]{todonotes}

\usepackage{enumitem}
\usepackage{verbatim}

\newtheorem{theorem}{Theorem}[section]
\newtheorem*{theorem*}{Theorem}
\newtheorem{definition}[theorem]{Definition}

\theoremstyle{plain}
\newtheorem{corollary}[theorem]{Corollary}
\newtheorem*{corollary*}{Corollary}
\newtheorem{lemma}[theorem]{Lemma}
\newtheorem{proposition}[theorem]{Proposition}

\newcounter{mt}

\newtheorem{MainTheorem}[mt]{Theorem}
\newtheorem{MainCorollary}[mt]{Corollary}

\newtheorem*{lemma*}{Lemma}
\newtheorem*{question*}{Question}

\theoremstyle{definition}

\makeatletter
\newcommand{\tpitchfork}{%
  \vbox{
    \baselineskip\z@skip
    \lineskip-.52ex
    \lineskiplimit\maxdimen
    \m@th
    \ialign{##\crcr\hidewidth\smash{$-$}\hidewidth\crcr$\pitchfork$\crcr}
  }%
}
\makeatother

\BibSpec{article}{%
  +{}{\PrintAuthors}  		{author}
  +{,}{ \textit}     		{title}
  +{,}{ }             		{journal}
  +{}{ \textbf}       		{volume}
  +{}{ \parenthesize} 		{date}
  +{}{, no. } 		{number}
  +{,}{ }      	      		{conference}
  +{,}{ }      	      		{book}
  +{,}{ }            		{pages}
  +{,}{ }            	 	{note}
  +{,}{ }            	 	{status}
  +{,}{  \texttt } {eprint}
  +{.}{}              {transition}
}

\BibSpec{book}{%
  +{}{\PrintAuthors}  {author}
  +{,}{ \textit}      {title}
  +{,}{ }             {publisher}
  +{,}{ }             {place}
  +{,}{ }             {date}
  +{.}{}              {transition}
}

\BibSpec{incollection}{%
    +{}  {\PrintAuthors}                {author}
    +{,} { \textit}                     {title}
    +{.} { }                            {part}
    +{:} { \textit}                     {subtitle}
    +{,} { \PrintContributions}         {contribution}
    +{,} { \PrintConference}            {conference}
    +{,} { }                            {booktitle}
    +{,} { pp.~}                        {pages}
    +{,}{ }       		{series}
    +{}{ \textbf}       		{volume}
    +{,} { }                            {publisher}
    +{,} { }                 {date}
    +{,} { }                            {status}
    +{,} { \PrintDOI}                   {doi}
    +{,} { available at \eprint}        {eprint}
    +{}  { \parenthesize}               {language}
    +{}  { \PrintTranslation}           {translation}
    +{;} { \PrintReprint}               {reprint}
    +{.} { }                            {note}
    +{.} {}                             {transition}
}

\newcommand{\CC}{\mathbb{C}}

\newcommand{\NN}{\mathbb{N}}

\newcommand{\RR}{\mathbb{R}}

\newcommand{\ZZ}{\mathbb{Z}}

\newcommand{\ddt}{\left.\frac{d^2}{dt^2}\right|_{t=0}}

\DeclareMathOperator{\GL}{GL}
\DeclareMathOperator{\SL}{SL}

\DeclareMathOperator{\OO}{O}

\DeclareMathOperator{\supp}{supp}

\DeclareMathOperator{\linspan}{span}

\DeclareMathOperator{\interior}{int}
\DeclareMathOperator{\relint}{relint}

\newcommand{\DD}{\mathbb{D}}

\newcommand{\unitsurf}{\mathbb{S}}

\newcommand{\IntBody}{\mathrm{I}}
\newcommand{\JOp}{\mathrm{J}}

\newcommand{\convexbodies}{\mathcal{K}}
\newcommand{\starbodiesO}{\mathcal{S}_0}
\newcommand{\convexbodiesO}{\mathcal{K}_{(0)}}

\newcommand{\Vol}[2]{V_{#1}\left(#2\right)}
\newcommand{\CompParSec}[2]{\mathrm{A}^{\CC}_{#1,#2}}
\newcommand{\RealParSec}[2]{\mathrm{A}^{\RR}_{#1,#2}}

\newcommand{\ComplSphRad}{\mathcal{R}_c}

\newcommand{\ComplPosReMom}[0]{\mathcal{M}_{p,v}^{\Re, +}}

\title[]{Complex $L_p$-Intersection Bodies}
\author{Simon Ellmeyer}
\author{Georg C. Hofst\"atter}
\email{simon.ellmeyer@tuwien.ac.at}
\email{georg.hofstaetter@uni-jena.de}

\address{Institute f. Discrete Mathematics and Geometry, TU Wien, 1040 Wien, Austria}
\address{Institute f. Mathematics, Friedrich-Schiller-University Jena, 07743 Jena, Germany}

\date{\today}

\begin{document}

\begin{abstract}
Interpolating between the classical notions of intersection and polar centroid bodies, (real) $L_p$-intersection bodies, for $-1<p<1$, play an important role in the dual $L_p$-Brunn--Minkowski theory. Inspired by the recent construction of complex centroid bodies, a complex version of $L_p$-intersection bodies, with range extended to $p>-2$, is introduced, interpolating between complex intersection and polar complex centroid bodies. It is shown that the complex $L_p$-intersection body of an $\unitsurf^1$-invariant convex body is pseudo-convex, if \linebreak $-2<p<-1$ and convex, if $p\geq-1$. Moreover, intersection inequalities of Busemann--Petty type in the sense of Adamczak--Paouris--Pivovarov--Simanjuntak are deduced.
\end{abstract}

\maketitle

\section{Introduction}

For a star body $K$ in $\RR^n$, that is a compact set with continuous, positive radial function, which is star-shaped around the origin, the intersection body $\IntBody K$ was defined by Lutwak~\cite{Lutwak1988} as the unique, origin-symmetric star body satisfying
\begin{align}\label{eq:defIntbody}
	V_1(\IntBody K \cap \linspan^\RR\{u\}) = V_{n-1}(K \cap u^\perp), \qquad u \in \unitsurf^{n-1},
\end{align}
where $V_i$ denotes the $i$-dimensional volume, $\linspan^\RR\{u\}$ is the linear span of $u$, $u^\perp$ is the linear hyperplane orthogonal to $u$, and $\unitsurf^{n-1}$ is the Euclidean unit sphere.

While intersection bodies played a key role in the solution of the famous Busemann--Petty problem (see e.g.~\cite{Gardner1999} for an elegant unified solution and a comprehensive list of references), the origin of intersection bodies dates back to the pioneering works of Busemann on volume and area defined in Finsler spaces. Formulated in different terms, Busemann established his important convexity theorem~\cite{Busemann1949}, stating that the intersection body of an origin-symmetric convex body is convex, as well as his famous intersection inequality for convex bodies~\cite{Busemann1953}, which was extended to star bodies by Petty~\cite{Petty1961} in the following way: If $K$ is a star body in $\RR^n$, then
\begin{align}\label{eq:BusIntersectIneq}
	V_n(\IntBody K)/V_n(K)^{n-1} \leq \kappa_{n-1}^n/\kappa_n^{n-2}, 
\end{align}
where $\kappa_i = V_i(B^i)$ is the volume of the $i$-dimensional Euclidean unit ball $B^i$, and equality holds exactly for origin-symmetric ellipsoids.

Throughout the years, intersection bodies have sparked a lot of interest in a wide range of fields (see, e.g., \cites{Ludwig2006, Zhang1994, Goodey1995b,Gardner1994, Milman2006, Rubin2004, Zhang1996, Koldobsky1997, Koldobsky1998, Koldobsky1998b, Koldobsky1999, Koldobsky2000} for an overview). In particular, they played a central role in the development of the dual Brunn--Minkowski theory due to Lutwak~\cite{Lutwak1988}, where their special role was revealed by characterizations of intersection bodies from a valuation-theoretic point of view by Ludwig~\cite{Ludwig2006}. In the emerging $L_p$-Brunn--Minkowskiy theory and its dual, the concept was extended to the $L_p$-intersection body $\IntBody_{p} K$, defined for $K \in \starbodiesO(\RR^n)$, the set of all star bodies in $\RR^n$, and non-zero $p>-1$ by
\begin{align}\label{eq:defLpIntersectBody}
	\rho_{\IntBody_p K}(u)^{-p}=\int_K |\langle x,u\rangle|^p dx, \quad u \in \RR^n \setminus \{0\},
\end{align}
where $\rho_K(u) = \sup\{\lambda \geq 0: \lambda u \in K\}$, $u \in \RR^n \setminus\{0\}$, denotes the radial function of $K$ and $\langle \cdot, \cdot \rangle$ is the standard Euclidean inner product. Note that, for $p \geq 1$, this definition coincides (up to normalization) with the polar of the $L_p$-centroid body (first defined in \cite{Lutwak1997}, see also \cite{Gardner1999c}). For $-1<p<1$, $L_p$-intersection bodies were studied in \cites{Haberl2006, Haberl2008,Yaskin2006, Berck2009}. Using well-known properties of the $p$-cosine transform, the $L_p$-intersection bodies relate to the classical intersection body by 
\begin{align}\label{eq:convRealLpIntBody}
	\lim_{p \to -1^+} \left(\frac{1}{\Gamma(1+p)}\right)^{-1/p} \IntBody_p K  = 2\cdot \IntBody K, \quad K \in \starbodiesO(\RR^n),
\end{align}
see, e.g., \cites{Gardner1999c, Haberl2008, Grinberg1999}, where convergence is with respect to the radial metric on $\starbodiesO(\RR^n)$, that is, uniform convergence of radial functions on $\unitsurf^{n-1}$. As a natural consequence, $L_p$-analogues of classical problems for intersection bodies were considered, leading to fruitful interactions and many new results, including a convexity theorem by Berck~\cite{Berck2009} and characterizations by Haberl and Ludwig~\cite{Haberl2006}. The intersection inequality~\eqref{eq:BusIntersectIneq} was generalized in \cite{Lutwak1997} for $p \geq 1$, leading to the discovery of an interpolating family of inequalities between the polar Busemann--Petty centroid inequality ($p=1$) and the famous Blaschke--Santal\'o inequality ($p=\infty$). Very recently, this family of inequalities was extended in \cite{Adamczak2022} to $0<p<1$ and to $-1<p<0$ with $n/|p| \in \NN$, preceded by local inequalities including equality cases around the unit ball proved in \cite{Yaskin2022}.

\medskip

A different generalization of intersection bodies was recently introduced by Koldobsky, Paouris and Zymonopoulou~\cite{Koldobsky2013}, based on a Busemann--Petty-type problem in complex vector spaces, first considered in~\cite{Koldobsky2008}. Here, intersections by real hyperplanes are replaced by intersections by complex hyperplanes $u^{\perp, \CC}$ perpendicular to $u \in \unitsurf^{2n-1}$ with respect to the complex inner product on $\CC^n$, leading in a natural way to the definition of a complex intersection body. More precisely, identifying $\starbodiesO(\CC^n)$ with $\starbodiesO(\RR^{2n})$ as $\CC^n\cong\RR^{2n}$, for an $\unitsurf^1$-invariant star body $K \in \starbodiesO(\CC^n)$, that is, satisfying $cK = K$ for all $c \in \unitsurf^1 \subseteq \CC$, the complex intersection body $\IntBody_c K$ is defined as the unique $\unitsurf^1$-invariant star body satisfying
\begin{align}\label{eq:defComplIntBody}
	V_2(\IntBody_c K \cap \linspan^\CC\{u\}) = V_{2n-2}(K \cap u^{\perp, \CC}), \quad u \in \unitsurf^{2n-1},
\end{align}
where $\linspan^\CC\{u\}$ denotes the complex line defined by $u$. Note that, by $\unitsurf^1$-invariance, $\IntBody_c K \cap \linspan^\CC\{u\}$ is always a disk and \eqref{eq:defComplIntBody} determines its radius. Moreover, it was shown in \cite{Koldobsky2013} that $\IntBody_c K$ is convex whenever $K$ is the unit ball of a complex norm on $\CC^n$, that is, when $K \in \starbodiesO(\CC^n)$ is convex and $\unitsurf^1$-invariant.

\medskip

In this article, we combine the complex structure with the $L_p$-approach to define complex $L_p$-intersection bodies. For this reason, we adapt a strategy used by Abardia and Bernig~\cite{Abardia2011}, Abardia~\cite{Abardia2012} and Haberl~\cite{Haberl2019} who introduced complex projection, difference and centroid bodies, respectively. In order to state our main definition, let $\convexbodies(\CC)$ denote the set of convex bodies in $\CC$, that is, all compact and convex subsets of $\CC \cong \RR^2$, and let $h_K(u) = \sup\{\langle x, u\rangle: x \in K\}$ be the support function of $K \in \convexbodies(\CC)$. Replacing the support function $|\langle \cdot, u\rangle|$ of the interval $[-1,1]u$ by the support function of a convex body $Cu$, $C \in \convexbodies(\CC)$, then leads to the following definition.

\begin{definition}\label{def:compLpIntersectBody}
	Suppose that $C \in \convexbodies(\CC)$ contains the origin in its relative interior, $\dim C>0$, and $0\neq p \in (-\dim C,1)$. For $K \in \starbodiesO(\CC^n)$, the \emph{complex $L_p$-intersection body} $\IntBody_{C,p} K$ is the star body with radial function
	\begin{align}\label{eq:defCompLpIntersectBody}
		\rho_{\IntBody_{C,p}K}(u)^{-p}=\int_K h_{Cu}(x)^p dx, \quad u \in \unitsurf^{2n-1},
	\end{align}
	where $Cu = \{cu: c \in C\} \subseteq \CC^n$.
\end{definition}
\noindent
Note that for $C = [-1,1]$ we recover the (real) $L_p$-intersection bodies defined in \eqref{eq:defLpIntersectBody} and for $p=1$, this equals the polar complex centroid body introduced by Haberl in \cite{Haberl2019}. For $\dim C = 2$ the range of admissible values for $p$ extends to $(-2,-1]$, see Section~\ref{sec:complLpIntBodies} for details and some basic properties of complex $L_p$-intersection bodies. Let us also point out that in \cite{Wang2021} complex $L_p$-centroid (moment) bodies were defined in a similar way for $p\geq 1$.

\medskip

As our first main result, we show that complex $L_p$-intersection bodies interpolate between the polar complex centroid body ($p=1$) and the complex intersection body ($p=-2$), that is, we prove a complex analogue of \eqref{eq:convRealLpIntBody}, and thereby justify the name. To state the theorem, denote by $\convexbodiesO(\CC)$ the set of convex bodies $K \in \convexbodies(\CC)$ that contain the origin in their interior. Notice that $C \in \convexbodiesO(\CC)$ implies $\dim (C) = 2$.
\begin{MainTheorem}\label{mthm:limLpIntBodiesComplex}
	Suppose that $C \in \convexbodiesO(\CC)$. Then there exists $k_{C} > 0$, such that
	\begin{align*}
		\lim\limits_{p \to -2^+}\left(\frac{1}{\Gamma(p+2)}\right)^{-1/p}\IntBody_{C,p}K = k_C \cdot \IntBody_c\left(K^{\unitsurf^1}\right), 
	\end{align*}
	for every $K \in \starbodiesO(\CC^n)$, where $K^{\unitsurf^1} \in \starbodiesO(\CC^n)$ is the star body with radial function 
	\begin{align*}
		\rho_{K^{\unitsurf^1}}^{2n-2}(u)=\frac{1}{2\pi}\int_{\unitsurf^1}\rho_K^{2n-2}(cu)dc, \quad u \in \unitsurf^{2n-1}.
	\end{align*}
\end{MainTheorem}

As before, convergence is with respect to the radial metric. Similar to the real setting, Theorem~\ref{mthm:limLpIntBodiesComplex} is proved by showing by analytic continuation that a certain integral transform, used to define $\IntBody_{C,p}$, converges in the strong operator topology to a multiple of the complex spherical Radon transform (see Section~\ref{sec:prfThmAConv} for the definition), which defines $\IntBody_c$, as $p \to -2^+$. As a direct consequence, we obtain a simple formula for the multipliers of the complex spherical Radon transform, seen as a $\mathrm{U}(n)$-equivariant map on $C(\unitsurf^{2n-1})$, thereby partially recovering results (of higher generality) from Rubin \cite{Rubin2022} and showing that the complex intersection body map $\IntBody_c$, as well as the maps $\IntBody_{C,p}$, are injective on $\unitsurf^1$-invariant star bodies. See Section~\ref{sec:spherHarmonics} for the details of these calculations.

\medskip

Next, we consider an analogue of the well-known convexity theorem by Busemann~\cite{Busemann1949}, as well as the following extension by Berck~\cite{Berck2009} to (real) $L_p$-intersection bodies:
\begin{theorem}[\cite{Berck2009}]\label{thm:berckConvexityLpIntBody}
	Let $p>-1$ be non-zero. If $K \in \starbodiesO(\RR^n)$ is convex and origin-symmetric, then $\IntBody_p K$ is convex.
\end{theorem}
\noindent
Here, the condition that the body $K$ is origin-symmetric, that is, $K$ is the unit ball of a (real) norm cannot be omitted. Indeed, by translating the convex body $K$ one can show that the intersection body of a translate is not convex anymore (see \cite{Gardner2006}*{Thm.~8.1.8} and \cite{Brandenburg2023}). Transferring the symmetry condition to the complex setting, real norms are naturally substituted by complex norms, that is, origin-symmetry is replaced by $\unitsurf^1$-invariance, leading to the complex convexity theorem in \cite{Koldobsky2013} for $\IntBody_c$.
\begin{theorem}[\cite{Koldobsky2013}]\label{thm:koldConvexityComplIntBody}
	If $K \in \starbodiesO(\CC^n)$ is convex and $\unitsurf^1$-invariant, then $\IntBody_c K$ is convex.
\end{theorem}

\noindent
It is a natural question to ask whether complex $L_p$-intersection bodies are convex. As our second main result, we extend Theorems~\ref{thm:berckConvexityLpIntBody} and \ref{thm:koldConvexityComplIntBody} to complex $L_p$-intersection bodies of $\unitsurf^1$-invariant convex bodies in $\CC^n$, weakening for $-2<p<-1$ convexity to pseudo-convexity (see Section~\ref{sec:pseudoConv} for the definition). It is an interesting (open) question whether pseudo-convexity can be strengthened to convexity.

\begin{MainTheorem}\label{mthm:PlushIntBody}
	Suppose that $C \in \convexbodiesO(\CC)$. If $K \in \starbodiesO(\CC^n)$ is convex and $\unitsurf^1$-invariant, then $\interior \IntBody_{C,p} K$ is pseudo-convex, if $-2<p<-1$, and $\IntBody_{C,p} K$ is convex, if $-1 \leq p \neq 0$.
\end{MainTheorem}

\noindent
The proof of Theorem~\ref{mthm:PlushIntBody} is very much inspired by the techniques from \cite{Berck2009} and relies for $p>-1$ on Theorem~\ref{thm:berckConvexityLpIntBody}. Indeed, for $p>-1$, we actually show, using techniques from isometric embeddings into $L_p$-spaces, the following close relation.

\begin{MainTheorem}\label{prop:compLp=realLp}
	Suppose that $p>-1$ is non-zero and $C \in \convexbodiesO(\CC)$. Then there exists $d_{C,p} > 0$, such that $\IntBody_{C,p} K = d_{C,p}\IntBody_p K$ for every $\unitsurf^1$-invariant $K \in \starbodiesO(\CC^n)$.
\end{MainTheorem}
\noindent
Let us also note that, in general, Theorem~\ref{mthm:PlushIntBody} without the assumption of $\unitsurf^1$-invariance is false, if $p\geq-1$, as we show in Section~\ref{sec:counterExConv}.

\medskip

Turning now to inequalities for intersection bodies, our next main result relates the volume of complex $L_p$-intersection bodies with the volume of their real counterparts.
\begin{MainTheorem}\label{mthm:ineqComplexLpVsLp}
	Suppose that $C \in \convexbodiesO(\CC)$ is origin-symmetric and $-1\leq p < 1$ is non-zero. If $K \in \starbodiesO(\CC^n)$, then
	\begin{align}\label{eq:thmIneqComplLpVsLpClaimIntro}
		\Vol{2n}{\IntBody_{C,p}K}/\Vol{2n}{\IntBody_{C,p}B^{2n}} \leq \Vol{2n}{\IntBody_p K}/\Vol{2n}{\IntBody_{p}B^{2n}}.
	\end{align}
\end{MainTheorem}
\noindent
Let us note that the equality cases of \eqref{eq:thmIneqComplLpVsLpClaimIntro} can be completely described by a technical statement in terms of the convex body $C$ and will be stated later in Section~\ref{sec:ProofIntIneq}. By Theorem~\ref{prop:compLp=realLp}, clearly $\unitsurf^1$-invariant bodies satisfy equality. From Theorem~\ref{mthm:ineqComplexLpVsLp}, we deduce the following generalization of Busemann's intersection inequality~\eqref{eq:BusIntersectIneq} for $\IntBody_{C,p}$ leading to affine isoperimetric inequalities in the following sense. Here, we call an ellipsoid $E$ Hermitian, if $E = \varphi(B^{2n}) + t$ for $\varphi \in \GL(n,\CC), t \in \CC^n$.
\begin{MainCorollary}
	\label{mcor:CompIntIneq}
	Suppose that $C \in \convexbodiesO(\CC)$ is origin symmetric and $0<p<1$ or $-1 \leq p<0$ and $n/|p| \in \NN$.  Among $K \in \starbodiesO(\CC^n)$, the ratio
	\begin{align*}
		V_{2n}\left(\IntBody_{C,p}K\right)/V_{2n}\left(K\right)^{2n+p}
	\end{align*}
	is maximized by origin-symmetric Hermitian ellipsoids. If $p=-1$, these are the only maximizers.
\end{MainCorollary}

Indeed, Theorem~\ref{mthm:ineqComplexLpVsLp} shows that affine isoperimetric inequalities for real $L_p$ intersection bodies are stronger than their complex counterparts. Corollary~\ref{mcor:CompIntIneq} therefore follows directly from the very recent breakthrough in \cite{Adamczak2022}, where the following inequality for $L_p$-intersection bodies was proved using methods from stochastic geometry. 
\begin{theorem}[\cite{Adamczak2022}]
	\label{thm:lpIntIneqAdamczakEtAl}
	Suppose that $0<p<1$ or $-1<p<0$ and $n/|p| \in \NN$. 
	
	\noindent
	Among $K \in \starbodiesO(\RR^n)$, the ratio
	\begin{align*}
		V_n(\IntBody_{p} K)/V_n(K)^{n+p}
	\end{align*}
	is maximized by origin-symmetric ellipsoids.
\end{theorem}
\noindent Let us point out that Theorem~\ref{thm:lpIntIneqAdamczakEtAl} is a much deeper result than Theorem~\ref{mthm:ineqComplexLpVsLp}.



\section{Definition and Basic Properties of Complex $L_p$-Intersection Bodies}
\label{sec:complLpIntBodies}
In this section, we prove that by Definition~\ref{def:compLpIntersectBody} the complex $L_p$-intersection body map is well defined and show basic properties. We will deduce this from properties of a more general operator $\JOp_{C,p}$ on $C(\unitsurf^{2n-1})$. At first, we fix some notation and recall basic facts. Further background will be given in the section where it will be required first. As a general reference on convex bodies, we refer to the monographs by Gardner~\cite{Gardner2006} and Schneider~\cite{Schneider2014}.

\medskip

For a complex number $c \in \CC$, we write $\overline c$ for its complex conjugate, this also extends to $x \in \CC^n$. By identifying $\CC^n \cong \RR^{2n}$, the vector space $\CC^n$ can be endowed with the canonical inner product $\langle \cdot, \cdot \rangle$ on $\RR^{2n}$ and a complex inner product, which are related by
\begin{align*}
	x \cdot u = \langle x, u\rangle + i \langle -ix, u\rangle, \quad x, u \in \CC^n.
\end{align*}
Note that by our convention, $x \cdot (\lambda u) = \overline{\lambda} (x \cdot u)$, $x, u \in \CC^n$, $\lambda \in \CC$. Consequently, by identifying $\CC \cong \RR^2$, 
\begin{align}\label{eq:RelSkalProdCReal}
 \langle c, x \cdot u\rangle = \langle cu, x\rangle, \quad x,u \in \CC^n, c \in \CC.
\end{align}
The unit disk in $\CC$ is denoted by $\DD$. Recalling the definition of support functions, $h_K(u) = \sup\{\langle x, u\rangle: x \in K\}$ of $K \in \convexbodies(\CC^n)$, \eqref{eq:RelSkalProdCReal} directly implies for every $C \in \convexbodies(\CC)$,
\begin{align}\label{eq:hCuEqualhCudot}
 h_{Cu}(x) = h_C(x \cdot u), \quad x, u \in \CC^n.
\end{align}

\medskip
\noindent
For $K \in \starbodiesO(\CC^n)$, the complex parallel section function $\CompParSec{K}{u}$ is defined by
\begin{align}\label{eq:defCompParSec}
	\CompParSec{K}{u}(z) = V_{2n-2}(K \cap \{x \in \CC^n: x \cdot u = z\}), \quad u \in \CC^n\setminus\{0\}, z \in \CC.
\end{align}
Similarly, the real parallel section function $\RealParSec{K}{u}$ is defined using intersections by real (affine) hyperplanes. $\CompParSec{K}{u}$ can be written as complex Radon transform $\mathcal{R}^\CC_u[\mathbbm{1}_K]$ of the indicator function $\mathbbm{1}_K$ of $K$, where for $\psi \in C(\CC^n)$ with compact support,
\begin{align*}
	\mathcal{R}^\CC_u[\psi](z) = \int_{x \cdot u = z} \psi(x)dx, \quad u \in \CC^n \setminus\{0\}, z \in \CC.
\end{align*}
Moreover, by Fubini's theorem, $\CompParSec{K}{u}$ can be used to express certain integrals over parallel complex hyperplanes, that is,
\begin{align}\label{eq:propComplParSec}
	\int_K \varphi(x \cdot u) dx = \int_\CC \varphi(z) \CompParSec{K}{u}(z) dz
\end{align}
for every $\varphi \in C(\CC)$ and $u \in \CC^n\setminus\{0\}$.

\medskip

Suppose that $C \in \convexbodies(\CC)$ contains the origin in its relative interior, $\dim C > 0$ and let $p$ be non-zero with $p>-\dim C$. For every $f \in C(\unitsurf^{2n-1})$, we define $\JOp_{C,p} f$ by
\begin{align}\label{eq:defJOp}
 (\JOp_{C,p} f)(u) = \int_{\unitsurf^{2n-1}} h_{C}(v \cdot u)^p f(v) dv, \quad u \in \unitsurf^{2n-1}.
\end{align}
Rewriting Definition~\ref{def:compLpIntersectBody} in polar coordinates, shows that 
\begin{align}\label{eq:IntBodyByJOp}
\rho_{\IntBody_{C,p} K}^{-p} = \frac{1}{2n+p}\JOp_{C,p}(\rho_K^{2n+p}),
\end{align}
for every $K \in \starbodiesO(\CC^n)$.

\begin{lemma}\label{lem:JOpWellDef}
 $\JOp_{C,p}$ is a well-defined operator on $C(\unitsurf^{2n-1})$, which is Lipschitz continuous with Lipschitz-constant $\|\JOp_{C,p} 1\|_\infty$. Moreover, if $f \in C(\unitsurf^{2n-1})$ is strictly positive, so is $\JOp_{C,p} f$.
\end{lemma}
\begin{proof}
 First note that since $0 \in \relint C$, we have $h_C \geq 0$ and that $h_C(z) = 0$ if and only if $z$ is orthogonal to $\linspan^\RR C$. Hence, $h_{C}^p(v\cdot u)$ is well-defined and positive for all $v \in \unitsurf^{2n-1}$ that are not contained in the (proper) subspace defined by $v \cdot u \in (\linspan^\RR C)^{\perp_\RR}$, that is, for almost all $v \in \unitsurf^{2n-1}$, and we will interpret the integral in \eqref{eq:defJOp} accordingly. This readily implies that (assuming it is well-defined) $\JOp_{C,p} f$ is positive whenever $f$ is positive. 

 Next, we distinguish the cases $\dim C = 2$ and $\dim C = 1$. In the first case, $\dim C = 2$, since $0 \in \interior C$, there exist constants $d, D > 0$ such that $d \DD \subseteq C \subseteq D \DD$. A direct estimate then shows that 
 \begin{align*}
  |h_C(v \cdot u)^p f(v)| \leq \|f\|_\infty \max\{d^p, D^p\} h_\DD(v \cdot u)^p,
 \end{align*}
 that is, by dominated convergence, $\JOp_{C,p} f$ is well-defined and continuous whenever $v \mapsto h_\DD(v \cdot u)^p = |v \cdot u|^p$ is integrable with respect to the spherical Lebesgue measure on $\unitsurf^{2n-1}$ for some (and by invariance then every) $u \in \unitsurf^{2n-1}$. For this reason, let $u \in \unitsurf^{2n-1}$ be arbitrary and compute using polar coordinates (in $\CC^n$) and \eqref{eq:propComplParSec},
 \begin{align*}
  \int_{\unitsurf^{2n-1}} |v \cdot u|^p dv = (2n+p) \int_{B^{2n}} |x \cdot u|^p dx = (2n+p) \int_\CC |z|^p \CompParSec{B^{2n}}{u}(z) dz.
 \end{align*}
 As $\CompParSec{B^{2n}}{u}$ is bounded by some $M > 0$ and has compact support contained in some ball $R\DD$, $R>0$, both uniformly in $u$, the latter integral can be estimated, using polar coordinates (in $\CC$), by
 \begin{align*}
  (2n+p) \int_\CC |z|^p \CompParSec{B^{2n}}{u}(z) dz \leq (2n+p) M \int_{R\DD} |z|^p dz = (2n+p)M 2\pi \int_0^R r^{p+1} dr,
 \end{align*}
 which is finite since $p + 1 > -\dim C + 1 = -1$. We conclude that $\JOp_{C,p} f \in C(\unitsurf^{2n-1})$, whenever $\dim C = 2$.

 In the second case, $\dim C = 1$, there exists an origin-symmetric interval $I\subseteq \CC$ and constants $d, D > 0$ such that $d I \subseteq C \subseteq DI$, which reduces the claim to a similar calculation as in the previous case.
 
 Finally, Lipschitz-continuity of $\JOp_{C,p}$ follows by a direct estimate, the Lipschitz constant is given by $\|\JOp_{C,p} 1\|_\infty$.
\end{proof}
\medskip

\noindent
Indeed, the operators $\JOp_{C,p}f$ are jointly continuous in $C$ and $f$. Before stating and proving this explicitly, we give two technical lemmas required in the proof.
\begin{lemma}\label{lem:CalcJOp1}
 Suppose that $C \in \convexbodies(\CC)$ with $C \neq \{0\}$ and $0 \in \relint C$, and let $p > -\dim C$. Then there exists $c(n,p) > 0$ such that
 \begin{align*}
  (\JOp_{C,p} 1)(u) = c(n,p) \int_{\unitsurf^1} h_C(v)^p dv, \quad u \in \unitsurf^{2n-1}.
 \end{align*}
\end{lemma}
\begin{proof}
 A direct calculation using polar coordinates and \eqref{eq:propComplParSec} for the complex parallel section function $\CompParSec{B^{2n}}{u}$ yields for $u \in \unitsurf^{2n-1}$,
 \begin{align*}
  (\JOp_{C,p}1)(u) &= (2n+p) \int_{B^{2n}}h_C(x \cdot u)^p dx = (2n+p) \int_\CC h_C(z)^p \CompParSec{B^{2n}}{u}(z) dz \\
  &= (2n+p) \int_{\unitsurf^1} h_C(v)^p \int_0^\infty r^{p+1} \CompParSec{B^{2n}}{u}(rv) dr dv \\
  &= (2n+p)\int_0^\infty r^{p+1} \CompParSec{B^{2n}}{u}(r) dr \int_{\unitsurf^1} h_C(v)^p dv,
 \end{align*}
 where we used that $\CompParSec{B^{2n}}{u}(rv) = \CompParSec{B^{2n}}{u}(r)$ by the $\unitsurf^1$-invariance of $B^{2n}$.
\end{proof}

\noindent
In the following lemma, convergence of convex bodies is, as always, in the Hausdorff-topology, that is, uniform convergence on $\unitsurf^{2n-1}$ of support functions.

\begin{lemma}\label{lem:ConvSeqDualMixVolUnifBnd}
 Suppose that $(C_j)_{j \in \NN} \subseteq \convexbodiesO(\CC)$ converges to $C_0 \in \convexbodies(\CC)$, with $C_0 \neq\{0\}$ and $0 \in \relint C_0$, and let $p>-\min\{\dim C_j: j=0, 1, \dots\}$ be non-zero. Then there exists $M > 0$ such that
 \begin{align}\label{eq:lemConvSeqDualMixVolUnifBnd}
  \int_{\unitsurf^1} h_{C_j}(u)^p du < M, \quad j \in \NN.
 \end{align}
\end{lemma}
\begin{proof}
 First note that the integral in \eqref{eq:lemConvSeqDualMixVolUnifBnd} is always finite as the case $n=1$ of the previous Lemma~\ref{lem:JOpWellDef} shows, that is, it remains to show that the integral can be uniformly bounded when $j$ is large enough.
 
 If $\dim C_0 = 2$, there exist $a, b > 0$ such that $a\DD \subseteq C_j, C_0 \subseteq b \DD$, and a direct estimate shows the claim. We are therefore left to prove the claim for $\dim C_0 = 1$.
 To this end, observe that the convergence $C_j \to C_0$ implies that $C_j \cap (-C_j) \to C_0 \cap (-C_0)$, as $j \to \infty$, and let $2d_0$ be the length of the maximal, origin-symmetric interval that is contained in $C_0 \cap (-C_0)$. As $0 \in \relint C_0$, $d_0 > 0$. Since 
 \begin{align*}
  d_0 = \max_{u \in \unitsurf^1} h_{C_0 \cap (-C_0)}(u),
 \end{align*}
 the convergence of $C_j\cap(-C_j)$ implies that for $j$ sufficiently large, every $C_j \cap (-C_j)$ (and thus every $C_j$) contains an origin symmetric interval of length greater or equal $2d_0 - d_0 = d_0$.
 
 Moreover, since $C_j$ is a convergent sequence, there exists $D > 0$ such that $C_j \subseteq D\DD$ for all $j \in \NN$. Consequently, we have shown that, for every $j$ large enough there exists $\xi_j \in \unitsurf^1$, such that $[-\frac{d_0}{2} \xi_j, \frac{d_0}{2} \xi_j] \subseteq C_j \subseteq D \DD$, which implies that,
 \begin{align}\label{eq:bndSuppFcthochP}
  h_{C_j}(u)^p \leq \max\left\{\left(\frac{d_0}{2}\right)^p|\langle\xi_j, u\rangle|^p, D^p|u|^p\right\}, \quad u \in \unitsurf^1.
 \end{align}
 Therefore the claim follows from the integrability of $|\cdot|^p$ and the fact that the (finite) integral of $|\langle w, \cdot\rangle|^p$ does not depend on the choice of $w \in \unitsurf^{2n-1}$.
\end{proof}

\noindent We are now in a position to prove the aforementioned joint continuity of $\JOp_{C,p}$.

\begin{proposition}\label{prop:JOpJointlyCont}
 Suppose that $p>-2$. Then the map
 \begin{align*}
  \JOp: \{ C \in \convexbodies(\CC): C \neq \{0\}, 0 \in \relint C, \dim C > -p \} \times C(\unitsurf^{2n-1}) \to C(\unitsurf^{2n-1}),
 \end{align*}
 defined by $(C,f) \mapsto \JOp_{C,p} f$, is jointly continuous.
\end{proposition}
\begin{proof}
 Suppose that $C_j \to C$, for $\{0\} \neq C_j, C \in \convexbodies(\CC)$ with $0 \in \relint C_j, C$ and $\dim C_j, C > -p$, and that $f_j \to f$ uniformly, $f_j,f \in C(\unitsurf^{2n-1})$, as $j \to \infty$. We need to show that $\JOp_{C_j,p} f_j \to \JOp_{C,p} f$ uniformly on $\unitsurf^{2n-1}$ as $j \to \infty$. To this end, we will first show pointwise convergence of $\JOp_{C_j,p} f_j$ and then use the Arzel\`a-Ascoli theorem to deduce uniform convergence.
 
 Therefore, letting $u \in \unitsurf^{2n-1}$, a direct estimate yields
 \begin{align*}
  |(\JOp_{C_j,p} f_j - \JOp_{C,p}f)(u)| &\leq |(\JOp_{C_j,p} f_j - \JOp_{C_j,p}f)(u)| + |(\JOp_{C_j,p}f - \JOp_{C,p}f)(u)| \\
                                        &\leq \|f_j - f\|_\infty |(\JOp_{C_j,p} 1)(u)| + \|f\|_\infty \int_{\unitsurf^{2n-1}}\!\!\! |(h_{C_j}^p - h_C^p)(v \cdot u)| dv.
 \end{align*}
 By Lemmas~\ref{lem:CalcJOp1} and \ref{lem:ConvSeqDualMixVolUnifBnd} the first term on the right-hand side is bounded by \linebreak$M' \|f_j - f\|_\infty$, where $M'>0$ is some constant independent of $j$. Moreover, arguing as in the proof of Lemma~\ref{lem:ConvSeqDualMixVolUnifBnd}, see \eqref{eq:bndSuppFcthochP}, the integrand in the second term has an integrable majorant. The uniform convergence of $f_j$ and dominated convergence therefore imply that $\JOp_{C_j,p} f_j(u) \to \JOp_{C,p} f(u)$.
 
 Next, since $(\JOp_{C_j, p} f_j)(u)$ is convergent for every $u \in \unitsurf^{2n-1}$, the sequence is uniformly bounded, that is, the family $(\JOp_{C_j,p} f_j)_{j \in \NN}$ is pointwise bounded. In order to show equicontinuity, fix some arbitrary $u \in \unitsurf^{2n-1}$ and let $\eta \in \mathrm{U}(n)$ be a uni\-tary linear map. The invariance of the Lebesgue measure on $\unitsurf^{2n-1}$ then yields,
 \begin{align*}
  (\JOp_{C_j, p}f_j)(\eta u) = \int_{\unitsurf^{2n-1}}h_{C_j}((\eta^{-1} v) \cdot u)^p f_j(v) dv = \int_{\unitsurf^{2n-1}}h_{C_j}(v \cdot u)^p f_j(\eta v) dv.
 \end{align*}
 Letting $\varepsilon > 0$ arbitrary, by the equicontinuity of the $f_j$ on the compact set $\unitsurf^{2n-1}$, there exists an open neighborhood $U$ of the identity in $\mathrm{U}(n)$ such that \linebreak$|f_j(v) - f_j(\eta v)| < \varepsilon$ for all $v \in \unitsurf^{2n-1}$, $\eta \in U$, $j \in \NN$. Consequently, for all $\eta \in U$,
 \begin{align*}
  |(\JOp_{C_j, p}f_j)(u) - (\JOp_{C_j, p}f_j)(\eta u)| \leq \int_{\unitsurf^{2n-1}}\!\!\! h_{C_j}(v \cdot u)^p |f_j (v) - f_j(\eta v)| dv
  \leq \varepsilon |(\JOp_{C_j,p}1)(u)|,
 \end{align*}
 which, by the previous estimate $|(\JOp_{C_j,p}1)(u)| < M'$ (independently of $j$) and since $\{\eta u: \eta \in U\}$ is an open neighborhood of $u \in \unitsurf^{2n-1}$ shows the equicontinuity of the family $(\JOp_{C_j,p} f_j)_{j \in \NN}$.
 
 The Arzel\`a-Ascoli theorem thus implies the existence of a uniformly convergent subsequence $(\JOp_{C_{j_k},p} f_{j_k})_{k \in \NN}$. As the original sequence converges pointwise to $\JOp_{C,p} f$, we obtain $\JOp_{C_{j_k},p} f_{j_k} \to \JOp_{C,p} f$, and a standard argument (that is, starting with an arbitrary subsequence) implies that $\JOp_{C_{j},p} f_{j} \to \JOp_{C,p} f$, which completes the proof.
\end{proof}

\medskip

\noindent
Note that, for $C \in \convexbodiesO(\CC)$, Proposition~\ref{prop:JOpJointlyCont} can be proved directly by showing local Lipschitz-continuity of $\JOp_{C,p} f$ as a function in $C$.

\medskip

\noindent It follows now directly that the complex $L_p$-intersection body body is well defined and continuous.

\begin{corollary}\label{cor:CIntBodyCont}
Suppose that $p>-2$. Then the map
 \begin{align*}
  \IntBody: \{ C \in \convexbodies(\CC): C \neq \{0\}, 0 \in \relint C, \dim C > -p \} \times \starbodiesO(\CC^n) \to \starbodiesO(\CC^n),
 \end{align*}
 defined by $(C,K) \mapsto \IntBody_{C,p} K$, is well-defined and jointly continuous.
\end{corollary}
\begin{proof}
 This follows directly from \eqref{eq:IntBodyByJOp}, Lemma~\ref{lem:JOpWellDef}, Proposition~\ref{prop:JOpJointlyCont} and the fact that the maps $t \mapsto t^{2n+p}$ and $t \mapsto t^{-1/p}$ are locally Lipschitz-continuous for $t>0$. Note that Lemma~\ref{lem:JOpWellDef} asserts that $(\JOp_{C,p}\rho_K^{2n+p})^{-1/p}$ is positive and continuous and therefore a radial function of a star body in $\starbodiesO(\CC^n)$.
\end{proof}

\medskip
Note that the proofs of Lemma~\ref{lem:JOpWellDef} and Corollary~\ref{cor:CIntBodyCont} imply that for fixed $C \in \convexbodies(\CC)$ and non-zero $p>-\dim C$, the operator $\IntBody_{C,p}:\starbodiesO(\CC^n) \to \starbodiesO(\CC^n)$ is locally Lipschitz-continuous.

\medskip

In view of its importance for the real $L_p$-intersection body (see \cites{Haberl2006,Ludwig2006}), we close the section with the following corresponding property for complex $L_p$-intersection bodies. The proof is a direct computation and will be omitted.
\begin{lemma}
 Suppose that $C \in \convexbodies(\CC)$ contains the origin in its relative interior and let $p>-\dim C$ be non-zero. Then $\IntBody_{C,p}:\starbodiesO(\CC^n) \to \starbodiesO(\CC^n)$ is a $\GL(n,\CC)$-contravariant valuation with respect to $L_{-p}$-radial addition, that is,
 \begin{align*}
  \rho_{\IntBody_{C,p} (K \cup L)}^{-p} + \rho_{\IntBody_{C,p} (K \cap L)}^{-p} = \rho_{\IntBody_{C,p} (K)}^{-p} + \rho_{\IntBody_{C,p} (L)}^{-p}, \quad K, L \in \starbodiesO(\CC^n),
 \end{align*}
 and 
 \begin{align*}
  \IntBody_{C,p}(\Theta K) = |\det \Theta|^{-2/p} \Theta^{-\ast} \IntBody_{C,p}(K), \quad K \in \starbodiesO(\CC^n), \Theta \in \GL(n,\CC),
 \end{align*}
 where $\Theta^{-\ast} = (\Theta^\ast)^{-1}$ denotes the inverse of the Hermitian adjoint $\Theta^\ast = \overline{\Theta}^T$.
\end{lemma}

\medskip
\section{Proof of Theorem~\ref{mthm:limLpIntBodiesComplex} and Injectivity}
\subsection{Proof of Theorem~\ref{mthm:limLpIntBodiesComplex}}\label{sec:prfThmAConv}
In this section, we will use several results from the previous section to give a proof of Theorem~\ref{mthm:limLpIntBodiesComplex}, that is, to compute the limit of (a normalization of) $\IntBody_{C,p} K$ for $p \to -2^+$, where $C \in \convexbodiesO(\CC)$ and $K \in \starbodiesO(\CC^n)$. To this end, we will first show a similar result for the operator $\JOp_{C,p}$ and then deduce from it Theorem~\ref{mthm:limLpIntBodiesComplex}.

\medskip

A key ingredient of the proof of the statements in this section will be the well-known fact that the familiy of distributions $r^q_+$,
 \begin{align}\label{eq:defDistrtqplus}
  \phi \mapsto \langle r^q_+, \phi \rangle = \int_0^\infty r^q \phi(r) dr
 \end{align}
 is analytic for every $q \in \CC$ with $\Re q >-1$, and admits a meromorphic extension, with poles at $-\NN$ (see, e.g., \cite{Gelfand1964}*{Sec.~3.2}). Consequently,
 \begin{align*}
  \lim_{q\to 0} \int_0^\infty r^{q} \phi(r) dr = \int_0^\infty \phi(r) dr
 \end{align*}
 and, as can be directly checked,
 \begin{align}\label{eq:defDistrtqplusConvMin1}
  \lim_{q\to -1^+} \frac{1}{\Gamma(q+1)}\int_0^\infty r^{q} \phi(r) dr = \phi(0),
 \end{align}
 for every Schwartz function $\phi$ on $\RR$. Moreover, since all distributions $r^q_+$, $\Re q > -1$, and their limit distribution can be applied to continuous functions with compact support, \eqref{eq:defDistrtqplusConvMin1} holds for all $\phi \in C(\RR)$ with compact support (see, e.g., \cite{Hoermander2003}*{Thm.~2.1.8}).
 
 \medskip

\noindent
As the following proposition shows, the normalized operators $\JOp_{C, p}$ converge to a multiple of the \emph{complex spherical Radon transform} $\ComplSphRad$,
\begin{align*}
	(\ComplSphRad f)(u) = \int_{\unitsurf^{2n-1} \cap \{v \cdot u = 0\}} \!\!\!\!\!f(v) dv, \quad u \in \unitsurf^{2n-1},
\end{align*}
where $f \in C(\unitsurf^{2n-1})$.

\begin{proposition}\label{prop:ConvJOpMin2}
 Suppose that $C \in \convexbodiesO(\CC)$. Then there exists $k_C' > 0$ such that $\frac{1}{\Gamma(p+2)}\JOp_{C,p}$ converges to $k_C' \ComplSphRad$ in the strong operator topology, as $p \to -2^+$, that is,
 \begin{align}\label{eq:thmConvJOpMin2}
  \frac{1}{\Gamma(p+2)}\JOp_{C,p} f \to k_C' \ComplSphRad f, \quad p \to -2^+,
 \end{align}
 uniformly on $\unitsurf^{2n-1}$ for every $f \in C(\unitsurf^{2n-1})$.
\end{proposition}
\begin{proof}
 Suppose that $C \in \convexbodiesO(\CC)$ and $f \in C(\unitsurf^{2n-1})$. In order to prove \eqref{eq:thmConvJOpMin2}, we will first show that $\frac{1}{\Gamma(p+2)}\JOp_{C,p} f$ converges pointwise on $\unitsurf^{2n-1}$ and then use the Arzel\`a-Ascoli theorem to deduce uniform convergence.
 
 To this end, we use polar coordinates (in $\CC^n$), Fubini's theorem and again polar coordinates (in $\CC$) to rewrite $\JOp_{C,p} f(u)$ for $u \in \unitsurf^{2n-1}$,
 \begin{align*}
  (\JOp_{C,p} f)(u) &= (2n+p) \int_{B^{2n}\setminus\{0\}} h_C(x \cdot u)^p f\left(\frac{x}{\|x\|}\right) dx\\
                    &= (2n+p) \int_\CC h_C(z)^p \int_{x \cdot u = z} f\left(\frac{x}{\|x\|}\right) \mathbbm{1}_{B^{2n}\setminus\{0\}}(x) dx dz\\
                    &= (2n+p) \int_0^\infty r^{p+1} \int_{\unitsurf^1} h_C(v)^p\int_{x \cdot u = rv} f\left(\frac{x}{\|x\|}\right) \mathbbm{1}_{B^{2n}\setminus\{0\}}(x) dx dv dr.
 \end{align*}
 Letting $g_{u,v}(r) = \int_{x \cdot u = rv} f\left(\frac{x}{\|x\|}\right) \mathbbm{1}_{B^{2n}\setminus\{0\}}(x) dx$, and using again Fubini's theorem, we arrive at
 \begin{align}\label{eq:prfConvJOpMin2PreLim}
  (\JOp_{C,p} f)(u) = (2n+p) \int_{\unitsurf^1} h_C(v)^p \int_0^\infty r^{p+1} g_{u,v}(r) dr dv.
 \end{align}
 Next, noting that $g_{u,v}$ is continuous (by dominated convergence) and has compact support, we deduce by \eqref{eq:defDistrtqplusConvMin1},
 \begin{align*}
  \lim_{p \to -2^+} \frac{1}{\Gamma(p+2)} \int_0^\infty r^{p+1} g_{u,v}(r) dr = g_{u,v}(0)
 \end{align*}
 for every $v \in \unitsurf^1$ and $u \in \unitsurf^{2n-1}$. Consequently, the integrand in \eqref{eq:prfConvJOpMin2PreLim}, normalized by $\Gamma(p+2)$, converges pointwise to $h_C(v)^{-2} g_{u,v}(0)$. As there exists $d \in (0,1)$ such that $d\DD \subseteq C$, that is, $h_C(v)^p \leq d^p \leq d^{-2}$ for every $v \in \unitsurf^1$ and $-2<p<0$, and 
 \begin{align*}
  \frac{1}{\Gamma(p+2)}\int_0^\infty r^{p+1} |g_{u,v}(r)| dr &\leq \frac{\|f\|_\infty}{\Gamma(p+2)} \int_0^1 r^{p+1} \int_{x \cdot u = rv} \mathbbm{1}_{B^{2n}\setminus\{0\}}(x) dxdr \\
  &\leq \frac{\|f\|_\infty}{\Gamma(p+2)(p+2)} \kappa_{2n-2} = \frac{\|f\|_\infty}{\Gamma(p+3)} \kappa_{2n-2},
 \end{align*}
 where $\Gamma(p+3)$ is continuous for $p \geq -2$, the integrand in \eqref{eq:prfConvJOpMin2PreLim} is bounded uniformly in $p$. Dominated convergence thus implies that
 \begin{align*}
  \lim_{p \to -2^+}\frac{1}{\Gamma(p+2)}(\JOp_{C,p} f)(u) &= (2n-2) \int_{\unitsurf^1} h_C(v)^{-2} g_{u,v}(0) dv \\
   &= (2n-2) \int_{\unitsurf^1} h_C(v)^{-2} dv \int_{x \cdot u = 0} f\left(\frac{x}{\|x\|}\right) \mathbbm{1}_{B^{2n}\setminus\{0\}}(x) dx.
 \end{align*}
 Letting $k_C' = \int_{\unitsurf^1} h_C(v)^{-2} dv = 2 V_2(C^\circ)$ and using polar coordinates in $x \cdot u = 0$, the latter expression is equal to
 \begin{align*}
  k_C' (2n-2) \int_{\unitsurf^{2n-1} \cap \{v \cdot u = 0\}} \!\!\!\!\!f(v) dv \int_0^1 r^{2n-3} dr = k_C' (\ComplSphRad f)(u),
 \end{align*}
 that is, $\frac{1}{\Gamma(p+2)}(\JOp_{C,p} f)(u) \to k_C' (\ComplSphRad f)(u)$, $u \in \unitsurf^{2n-1}$, $p \to -2^+$, as claimed.

 Next, since $\frac{1}{\Gamma(p+2)}|(\JOp_{C,p} f)(u)|$ is convergent for every $u \in \unitsurf^{2n-1}$, the sequence is bounded, that is, the family $(\frac{1}{\Gamma(p+2)}\JOp_{C,p} f)_{p > -2}$ is pointwise bounded. In order to show equicontinuity, we proceed as in the proof of Proposition~\ref{prop:JOpJointlyCont} to conclude that for every $\varepsilon > 0$ and $u \in \unitsurf^{2n-1}$ there exists an open neighborhood $U$ of $u$ such that
 \begin{align*}
  |(\JOp_{C, p}f)(u) - (\JOp_{C, p}f)(w)| \leq \varepsilon |(\JOp_{C,p}1)(u)|, \quad w \in U.
 \end{align*}
 Hence, since $\frac{1}{\Gamma(p+2)}|(\JOp_{C,p} 1)(u)|$ is convergent (for $p\to -2^+$) and thus bounded, the family $(\frac{1}{\Gamma(p+2)}\JOp_{C,p} f)_{p > -2}$ is equicontinuous.

 The Arzel\`a-Ascoli theorem therefore implies the existence of a uniformly convergent subsequence, which, by pointwise convergence, must converge to $k_C' \ComplSphRad f$. A standard argument, finally, shows the uniform convergence of the whole sequence, which completes the proof.
\end{proof}

\medskip

\noindent Theorem~\ref{mthm:limLpIntBodiesComplex} is now a consequence of Proposition~\ref{prop:ConvJOpMin2}, since, by polar coordinates and $\unitsurf^1$-invariance, the radial function of the complex intersection body $\IntBody_c K$ satisfies
\begin{align}\label{eq:defCIntBodyByComplSphRad}
	\rho_{\IntBody_c K}(u) = \left( \frac{1}{(2n-2)\pi} \ComplSphRad \rho_K^{2n-2}(u) \right)^{1/2}, \quad u \in \unitsurf^{2n-1}.
\end{align}

\medskip

\begin{proof}[Proof of Theorem~\ref{mthm:limLpIntBodiesComplex}]
 First observe that Proposition~\ref{prop:ConvJOpMin2} readily implies that whenever $f_p \to f$ uniformly as $p \to -2^+$, $f_p, f \in C(\unitsurf^{2n-1})$, then $\frac{1}{\Gamma(p+2)}\JOp_{C,p} f_p$ converges uniformly to $k_C' \ComplSphRad f$. Indeed,
 \begin{align*}
  \left\| \frac{\JOp_{C,p} f_p}{\Gamma(p+2)} - k_C' \ComplSphRad f\right\|_\infty &\leq \frac{\|\JOp_{C,p} (f_p - f)\|_\infty}{\Gamma(p+2)} + \left\| \frac{\JOp_{C,p} f}{\Gamma(p+2)} - k_C' \ComplSphRad f\right\|_\infty \\
  &\leq \|f_p - f\|_\infty \frac{\|\JOp_{C,p}1\|_\infty}{\Gamma(p+2)} + \left\| \frac{\JOp_{C,p} f}{\Gamma(p+2)} - k_C' \ComplSphRad f\right\|_\infty,
 \end{align*}
 where the right-hand side converges to zero by the uniform convergence of $f_p$ to $f$ and since $\frac{\|\JOp_{C,p}1\|_\infty}{\Gamma(p+2)}$ is bounded by Proposition~\ref{prop:ConvJOpMin2} (for the first summand), and by Proposition~\ref{prop:ConvJOpMin2} (for the second summand).
 
 Next, note that for $K \in \starbodiesO(\CC^n)$ there exist $d> 0$ and $D > 1$ such that \linebreak$d < \rho_K(u) < D$ for all $u \in \unitsurf^{2n-1}$. Since the map $p \mapsto t^{2n+p}$, $t>0$, is differentiable with derivative $t^{2n+p}\ln(t)$, the mean value theorem of calculus implies for $-2<p<0$ and $u \in \unitsurf^{2n-1}$ that
 \begin{align*}
  |\rho_K(u)^{2n+p} - \rho_K(u)^{2n-2}|& \leq \max_{q \in [-2,p]} \rho_K(u)^{2n+q} |\ln(\rho_K(u))| |p+2|\\
  &\leq D^{2n} \max\{ |\ln(d)|, |\ln(D)|\} |p+2|,
 \end{align*}
 that is, $\rho_K^{2n+p} \to \rho_K^{2n-2}$ uniformly as $p\to -2^+$. Hence, by \eqref{eq:IntBodyByJOp} and the first part of the proof,
 \begin{align*}
  \lim_{p \to -2^+} \frac{1}{\Gamma(p+2)}\rho_{\IntBody_{C,p} K}^{-p} = \lim_{p \to -2^+} \frac{\JOp_{C,p}\rho_K^{2n+p}}{(2n+p)\Gamma(p+2)} = \frac{k_C'}{2n-2} \ComplSphRad \rho_K^{2n-2}
 \end{align*}
 uniformly on $\unitsurf^{2n-1}$. Moreover, a direct estimate using $\rho_K(u) \in [d,D]$ shows that $\ComplSphRad \rho_K^{2n-2}(u) \in (2n-2)\kappa_{2n-2}[d^{2n-2}, D^{2n-2}]$. Consequently, by uniform convergence, there exist constants $d',D' > 0$ such that
 \begin{align*}
  d' < \frac{1}{\Gamma(p+2)}\rho_{\IntBody_{C,p} K}(u)^{-p} < D', \quad u \in \unitsurf^{2n-1},
 \end{align*}
 for all $p<0$ sufficiently close to $-2$. Repeating the above argument for the differentiable function $p \mapsto t^{-1/p}$, $t>0$, and using that the functions $t \mapsto t^{-1/p}$, $t \in [d',D']$ and $-2<p<-1$, are Lipschitz-continuous with Lipschitz constants uniformly bounded by $d'^{-1/2}$, then yields
 \begin{align}\label{eq:prfThmAConv}
  \lim_{p \to -2^+} \frac{\rho_{\IntBody_{C,p} K}}{\Gamma(p+2)^{-1/p}} = \left( \frac{k_C'}{2n-2} \ComplSphRad \rho_K^{2n-2} \right)^{1/2}
 \end{align}
 uniformly on $\unitsurf^{2n-1}$.

 Finally, as it is a direct computation that $\ComplSphRad f = \ComplSphRad f^{\unitsurf^1}$, where for $f \in C(\unitsurf^{2n-1})$,
 \begin{align*}
  f^{\unitsurf^1}(u) = \frac{1}{2\pi}\int_{\unitsurf^1} f(cu) dc, \quad u \in \unitsurf^{2n-1},
 \end{align*}
 and by \eqref{eq:defCIntBodyByComplSphRad}, the right-hand side of \eqref{eq:prfThmAConv} is equal to $(\pi k_C')^{1/2} \rho_{\IntBody_c K^{\unitsurf^1}}$, which completes the proof by setting $k_C = (\pi k_C')^{1/2}$.
\end{proof}

\subsection{Spherical Harmonics and Injectivity}\label{sec:spherHarmonics}
In this section, we will use spherical harmonics to show a criterion for the operators $\JOp_{C,p}$ to be injective and deduce that every $\JOp_{C,p}$ is injective on $\unitsurf^1$-invariant continuous functions. As a by-product, we will calculate the multipliers of $\JOp_{C,p}$ in terms of the Fourier coefficients of $h_C^p$, which leads (by taking limits) to a closed formula for the multipliers of the complex spherical Radon transform $\ComplSphRad$. All results for $\JOp_{C,p}$ directly translate to $\IntBody_{C,p}$.

\bigskip

Before stating and proving these results, we recall the required basic facts on spherical harmonics in complex vector spaces. We will follow the presentation in \cite{Abardia2015}, for further details we refer to the book by Groemer~\cite{Groemer1996} (for spherical harmonics in relation to convex geometry), to \cites{Quinto1987, Rudin2008}, as well as to the references therein.

First, recall that the space $\mathcal{H}^{2n}$ of spherical harmonics in $\unitsurf^{2n-1}$, that is, of harmonic polynomials on $\CC^n = \RR^{2n}$ restricted to $\unitsurf^{2n-1}$, naturally decomposes into $\OO(2n)$-irreducible subspaces,
\begin{align*}
 \mathcal{H}^{2n} = \bigoplus_{k=0}^\infty \mathcal{H}_k^{2n},
\end{align*}
where $\mathcal{H}_k^{2n}$ is the space of spherical harmonics that are homogeneous of degree $k\in \NN$. In presence of a complex structure, the spaces $\mathcal{H}_k^{2n}$ can be decomposed further into $\mathrm{U}(n)$-irreducible subspaces $\mathcal{H}_{k,l}^{2n}$ of spherical harmonics of bi-degree $(k,l)$. Here, a spherical harmonic $Y \in \mathcal{H}_{k,l}^{2n}$ has \emph{bi-degree} $(k,l) \in \NN \times \NN$, if $Y(cu) = c^k \overline{c}^l Y(u)$ for all $u \in \unitsurf^{2n-1}$ and $c \in \unitsurf^1$.

Denoting by $\pi_{k,l}$ the orthogonal projection from $L_2(\unitsurf^{2n-1})$ (endowed with the standard $L_2$-inner product) onto $\mathcal{H}_{k,l}^{2n}$, every $f \in C(\unitsurf^{2n-1})$ is uniquely determined by its harmonic components $\pi_{k,l} f \in \mathcal{H}_{k,l}^{2n}$, $k,l \in \NN$.

Next, fixing a point $\bar e \in \unitsurf^{2n-1}$, there exists a unique spherical harmonic \linebreak$\widetilde{P}_{k,l} \in \mathcal{H}_{k,l}^{2n}$, such that $\widetilde{P}_{k,l}(\bar e) = 1$ and $\widetilde{P}_{k,l}$ is invariant under the stabilizer \linebreak$\mathrm{U}(n-1) \subseteq \mathrm{U}(n)$ of $\bar e$. The existence of $\widetilde{P}_{k,l}$ and some properties of it, that we will need later on, are the content of the following proposition from \cite{Johnson1977}*{Thm.~3.1(3)}, see also \cite{Quinto1987}*{Prop.~4.2} for the formulation given here.

\begin{proposition}[\cites{Johnson1977,Quinto1987}]\label{prop:PropJacobiPoly}
 Let $k,l \in \NN$. Then $\mathcal{H}_{k,l}^{2n}$ contains a unique $\mathrm{U}(n-1)$-invariant spherical harmonic $\widetilde{P}_{k,l}$ with $\widetilde{P}_{k,l}(\bar e) = 1$, given by $\widetilde{P}_{k,l}(u) = P_{k,l}(\bar e \cdot u)$ for a polynomial $P_{k,l}: \DD \to \CC$, and satisfying
 \begin{enumerate}
  \item $P_{k,l}(\overline{z}) = \overline{P_{k,l}(z)}$, and
  \item $P_{k,l}(z) = z^{|k-l|} Q_{\min\{k,l\}}(|k-l|, n-2, |z|^2)$, for all $z \in \DD$,
 \end{enumerate}
 where $\{Q_l(a,b,\cdot): l \in \NN\}$ is the complete set of polynomials orthogonal on $[0,1]$ with respect to the $L_2$-inner product with weight $t^a(1-t)^b$ and $Q_l(a,b,1)=1$, $a,b>-1$.
\end{proposition}
\noindent The polynomial $P_{k,l}: \DD \to \CC$ is called \emph{Jacobi polynomial} of order $(k,l)$.

In analogy to their real counterparts (Legendre polynomials), Jacobi polynomials are very helpful in relation with transforms on $C(\unitsurf^{2n-1})$ given by a kernel $\phi$, as the following \emph{complex Funk--Hecke theorem} shows.

\begin{theorem}[\cite{Quinto1987}*{Thm.~4.4}]\label{thm:complFunkHecke}
Suppose that $\phi \in L_2(\DD, (1-|z|^2)^{n-2}dz)$ and let $Y_{k,l} \in \mathcal{H}_{k,l}^{2n}$. Then
\begin{align*}
 \int_{\unitsurf^{2n-1}} \phi(v \cdot u) Y_{k,l}(v) dv = \lambda_{k,l}[\phi] Y_{k,l}(u), \quad u \in \unitsurf^{2n-1},
\end{align*}
with
 \begin{align}\label{eq:thmComplFunkHeckeEqMult}
  \lambda_{k,l}[\phi] = (2n-2)\kappa_{2n-2} \int_\DD \phi(z) \overline{P_{k,l}}(z) (1-|z|^2)^{n-2} dz.
 \end{align}
\end{theorem}
\noindent 
In general, a transform $T: C(\unitsurf^{2n-1}) \to C(\unitsurf^{2n-1})$ that satisfies
\begin{align*}
 \pi_{k,l}(T f) = \lambda_{k,l}[T] \pi_{k,l} f, \quad f \in C(\unitsurf^{2n-1}),
\end{align*}
is called a \emph{multiplier transform} with multipliers $\lambda_{k,l}[T] \in \CC$. Note that since every $f \in C(\unitsurf^{2n-1})$ is completely determined by its projections $\pi_{k,l} f$, $k,l \in \NN$, a multiplier transform is injective if and only if all of its multipliers are non-zero.

\medskip

Examples of multiplier transforms are given by $\JOp_{C,p}$ (applying Theorem~\ref{thm:complFunkHecke}) and by the complex spherical Radon transform $\ComplSphRad$, as we will see later on. Another, very well-known example is the \emph{non-symmetric $L_p$-cosine transform} $C_p^+$, $p>-1$ is non-zero, given by $\phi(z) = (\Re z)_+^p$, where $t_+ = \max\{t,0\}$, that is,
\begin{align}\label{eq:asymmLpCosTransf}
 (C_p^+ f)(u) = \int_{\unitsurf^{2n-1} \cap u^+} |\langle v, u\rangle|^p f(v) dv, \quad u \in \unitsurf^{2n-1},
\end{align}
for every $f \in C(\unitsurf^{2n-1})$, writing $u^+ = \{ v\in \unitsurf^{2n-1}: \langle v, u\rangle \geq 0\}$.

The multipliers of $C_p^+$ as a transform on a real vector space were calculated by different means by Rubin~\cite{Rubin2000} (for dimension $3$ and higher) and Haberl~\cite{Haberl2008}*{Lem.~5} (also in dimension $2$). Since $\mathcal{H}_{k,l}^{2n} \subseteq \mathcal{H}_{k+l}^{2n}$, the multipliers of $C_p^+$ when viewed as a transform on a complex vector space are equal to the corresponding real multipliers, that is,
\begin{align}\label{eq:asymmLpCosTransfMult}
 \lambda_{k,l}[C_p^+] = \frac{\pi^n}{2^p}\frac{\Gamma(p+1)}{\Gamma\left(n + \frac{p+k+l}{2}\right)\Gamma\left(\frac{p-k-l}{2}+1\right)}
\end{align}
for $p>-1$ non-zero such that $p$ is not an integer. In particular, $\lambda_{k,l}[C_p^+]\neq 0$ for all $k,l \in \NN$, that is, $C_p^+$ is injective for $p \in (-1,\infty) \setminus \NN$.

\medskip

We are now ready to state the main proposition to prove injectivity of $\JOp_{C,p}$, calculating the multipliers of the transforms $\JOp_{C,p}$. In the statement of the proposition, we use the notation of the \emph{$k^{th}$ Fourier coefficient} $c_k(f)$ of $f \in C(\unitsurf^1)$,
\begin{align}\label{eq:defFourierCoeff}
c_0(f) = \frac{1}{2\pi} \int_{\unitsurf^1} f(c) dc \quad \text{ and } \quad c_k(f) = \frac{1}{\pi} \int_{\unitsurf^1}f(c) c^k dc, \quad k \in \ZZ\setminus\{0\}.
\end{align}

\begin{proposition}\label{prop:MultJOp}
 Suppose that $C \in \convexbodies(\CC)$, with $C \neq \{0\}$ and $0 \in \relint C$, and let $p > -\dim C$ be non-zero. Then the multipliers of the transform $\JOp_{C,p}$ are given for $k,l \in \NN$ by
 \begin{align}\label{eq:multJOp}
  \lambda_{k,l}[\JOp_{C,p}] = \begin{cases}
                               c_0(h_C^p) 2\alpha_{k,l}^{(n,p)}, & k = l,\\
                               c_{l-k}(h_C^p) \alpha_{k,l}^{(n,p)}, & k \neq l,
                              \end{cases}
 \end{align}
 where
 \begin{align*}
  \alpha_{k,l}^{(n,p)} = \pi^n\frac{\Gamma\left(\frac{p+k-l}{2}+1\right)\Gamma\left(\frac{p-k+l}{2}+1\right)}{\Gamma\left(\frac{p+k+l}{2}+n\right)\Gamma\left(\frac{p-k-l}{2}+1\right)}.
 \end{align*}

\end{proposition}
\begin{proof}
 By \eqref{eq:thmComplFunkHeckeEqMult}, we calculate using polar coordinates and the properties of Jacobi polynomials from Proposition~\ref{prop:PropJacobiPoly},
 \begin{align*}
  \frac{1}{(2n-2)\kappa_{2n-2}}\lambda_{k,l}[\JOp_{C,p}] = \int_{\unitsurf^1}\int_0^1 h_C(c)^p \overline{P_{k,l}}(rc) (1-r^2)^{n-2}r^{p+1} dr dc \\
  = \int_{\unitsurf^1} h_C(c)^p c^{l-k} dc \int_0^1 Q_{\min\{k,l\}}(|k-l|,n-2,r^2) (1-r^2)^{n-2}r^{p+1+|k-l|} dr,
 \end{align*}
 where the second integral does not depend on $C$ anymore. In particular, when $p>-1$, we can repeat the argument for $h_C^p(z)$ replaced by the kernel $h_{[-1,1]}^p(z) \mathbbm{1}_{\Re z \geq 0}$ of the non-symmetric $L_p$-cosine transform $C_p^+$ to obtain
 \begin{align*}
  \lambda_{k,l}[\JOp_{C,p}] = \frac{\int_{\unitsurf^1} h_C(c)^p c^{l-k} dc}{\int_{\unitsurf^1} (\Re c)^p \mathbbm{1}_{\Re c \geq 0} c^{l-k} dc} \lambda_{k,l}[C_p^+] = \frac{c_{l-k}(h_C^p)}{c_{l-k}((\Re c)^p \mathbbm{1}_{\Re c \geq 0})}\lambda_{k,l}[C_p^+].
 \end{align*}
 Next, \eqref{eq:asymmLpCosTransfMult} and direct computations using identities for the reciprocal beta function yield for $k \neq l$,
 \begin{align*}
  c_{l-k}((\Re c)^p \mathbbm{1}_{\Re c \geq 0}) = \frac{1}{\pi}\int_{-\pi/2}^{\pi/2} \!\!\!\!\!\cos(t)^p e^{i(l-k)t} dt = \frac{ \Gamma(p+1)}{2^p \Gamma\left(\frac{p+k-l}{2}+1\right)\Gamma\left(\frac{p-k+l}{2} + 1\right)}
 \end{align*}
 and for $k=l$,
 \begin{align*}
  c_0((\Re c)^p \mathbbm{1}_{\Re c \geq 0}) =\frac{1}{2\pi}\int_{-\pi/2}^{\pi/2}\!\!\!\!\!\cos(t)^p dt= \frac{ \Gamma(p+1)}{2^{p+1} \Gamma\left(\frac{p}{2} + 1\right)^2}.
 \end{align*}
 This proves the claim when $p>-1$. Noting, finally, that both sides of \eqref{eq:multJOp} are analytic functions in $p$ (for $\Re(p)>-\dim C$) that coincide on the set $(-1,0)$ and therefore on their domains, completes the proof. 
\end{proof}

Note that $\alpha_{k,l}^{(n,p)} \neq 0$ for all $k,l \in \NN$ and non-zero $p > -2$, $p \not \in \ZZ$, as the gamma function has no zeros and its poles are exactly the non-positive integers. In particular, we have shown the following.

\begin{corollary}
 Suppose that $C \in \convexbodies(\CC)$, with $C \neq \{0\}$ and $0 \in \relint C$, and let $p > -\dim C$, $p \not \in \ZZ$, be non-zero. Then $\JOp_{C,p}$ is injective if and only if $c_k(h_C^p) \neq 0$ for all $k \in \ZZ$.
\end{corollary}

\bigskip

\noindent We turn now to $\unitsurf^1$-invariant functions on $\unitsurf^{2n-1}$. Here, the computation simplifies to the case $C = \DD$, since
\begin{align}
2\pi(\JOp_{C,p} f^{\unitsurf^1}) (u) =&\int_{\unitsurf^{2n-1}}\!\!\! h_C(v \cdot u)^p \int_{\unitsurf^1}f(cv)dc dv \label{eq:S1invariantC=D1} \\
=& \int_{\unitsurf^{2n-1}}\int_{\unitsurf^1} h_C(\overline{c} (w \cdot u))^p dc f(w) dw=2 \pi (\JOp_{d\DD,p} f)(u) \label{eq:S1invariantC=D2}
\end{align}
for every $f \in C(\unitsurf^{2n-1})$, as the inner integral on the right-hand side can be written as $ h_{d\DD}(w \cdot u)^p$ for some $d > 0$ not depending on $f$. 

This can also be seen in terms of spherical harmonics, since restricting to \linebreak$\unitsurf^1$-invariant functions corresponds exactly to restricting to the spaces $\mathcal{H}_{k,k}^{2n}$, $k \in \NN$. Indeed, the definition of bi-degree directly implies that a function $f \in C(\unitsurf^{2n-1})$ is $\unitsurf^1$-invariant if and only if $\pi_{k,l} f = 0$, $k\neq l \in \NN$, see, e.g., \cite{Abardia2015}*{Lem.~4.8}. Consequently, $\JOp_{C,p}$ is completely determined on $\unitsurf^1$-invariant functions by $\lambda_{k,k}[\JOp_{C,p}]$, $k \in \NN$, that is, using \eqref{eq:multJOp}, by $c_0(h_C^p) \neq 0$. Thus, we conclude the following.
\begin{corollary}
 Suppose that $C \in \convexbodies(\CC)$, with $C \neq \{0\}$ and $0 \in \relint C$, and let $p > -\dim C$ be non-zero. Then $\JOp_{C,p}$ is injective on $\unitsurf^1$-invariant functions in $C(\unitsurf^{2n-1})$.
\end{corollary}

\bigskip

Finally, as we have seen in Section~\ref{sec:prfThmAConv}, letting $p \to -2^+$, the operators $\JOp_{C,p}$ converge appropriately normalized (in the strong operator topology) to the complex spherical Radon transform $\ComplSphRad$ for which $\ComplSphRad f = \ComplSphRad f^{\unitsurf^1}$ holds. Consequently, the multipliers of $\ComplSphRad$ can be directly calculated from \eqref{eq:multJOp} by taking the limit.

\begin{proposition}
 The multipliers of the complex spherical Radon transform $\ComplSphRad$ are given by $\lambda_{k,l}[\ComplSphRad]=0$ for $k \neq l$ and
 \begin{align*}
  \lambda_{k,k}[\ComplSphRad] = (-1)^{k} 2\pi^{n-1} \frac{k!}{(n+k-2)!}, \quad k \in \NN.
 \end{align*}
 In particular, the complex intersection body map $\IntBody_c$ is injective.
\end{proposition}

\section{(Pseudo-)Convexity}\label{sec:pseudoConv}
In this section we first collect the definition and basic properties of pseudo-convex sets that are used to prove Theorem~\ref{mthm:PlushIntBody}. As a general reference for pseudo-convex sets and plurisubharmonic functions, we refer to \cites{Hoermander1973, Krantz1992}.

Next, we prove Theorem~\ref{mthm:PlushIntBody} following the ideas of Berck~\cite{Berck2009} for his convexity theorem for $L_p$-intersection bodies. More precisely, we first establish concavity properties for complex $p$-moments of convex bodies using inequalities of Brunn--Minkowski type, which are then used to show that the reciprocal radial functions of $\IntBody_{C,p} K$ satisfy the sufficient conditions for pseudo-convexity in Theorem~\ref{thm:leviCondPseudoConvex}, where $K \in \convexbodiesO(\CC^n)$ is $\unitsurf^1$-invariant and has a smooth boundary. The general case then follows by approximation.

In the final part of this section we give examples in the range $-1<p<1$ of convex bodies $K$ that are not $\unitsurf^1$-invariant, such that $\IntBody_{C,p} K$ is not convex for some $C \in \convexbodiesO(\CC)$, showing that $\unitsurf^1$-invariance is a necessary condition. 


\subsection{Basic notions}

\noindent
First, recall that a function $\varphi: \Omega \to [-\infty, \infty)$, defined on an open subset $\Omega \subseteq \CC^n$, is called \emph{plurisubharmonic}, if
\begin{itemize}
	\item $\varphi$ is upper semi-continuous;
	\item for all $u, v \in \CC^n$, the map $z \mapsto \varphi(u + z v)$ is subharmonic where it is defined,
\end{itemize}
see, e.g., \cite{Hoermander1973}*{Def.~1.6.1 and 2.6.1}. Examples are given by all subharmonic and, hence, by all convex functions on $\CC^n$. Using this notion, pseudo-convex sets are defined as follows.

\begin{definition}[\cite{Hoermander1973}*{Def.~2.6.8}]\label{def:hartogPseudoConvex}
	An open, connected set $K \subseteq \CC^n$ is \emph{pseudo-convex}, if there exists a continuous, plurisubharmonic function $\varphi$ in $K$ such that the sets
	\begin{align*}
		\{z \in K: \varphi(z) < c\}, \quad c \in \RR,
	\end{align*}
	are all relatively compact in $K$.
\end{definition}

\noindent 
Note that this is also called \emph{Hartogs pseudo-convex} and equivalent to $K$ being a domain of holomorphy or holomorphically convex. For sets with more regular boundary, the \emph{Levi condition} yields an equivalent statement, which is more accessible: 

\begin{theorem}[\cite{Hoermander1973}*{Thm.~2.6.12}]\label{thm:leviCondPseudoConvex}
	Suppose that $K \subseteq \CC^n$ is an open set with \linebreak$C^2$-boundary, given by
	\begin{align*}
		K = \{u \in \CC^n: \rho(u)< 0\},
	\end{align*}
	where $\rho:\CC^n \to \RR$ is $C^2$ in a neighborhood of $\mathrm{cl}\, K$ and $\nabla \rho \neq 0$ on $\mathrm{bd}\, K$. Then $K$ is pseudo-convex, if and only if 
	\begin{align*}
		\Delta_z \rho(u+z v)|_{z = 0} \geq 0,
	\end{align*}
	for all $u \in \mathrm{bd}\, K$ and $v \in \CC^n$ with $\nabla \rho(u) \cdot v = 0$.
\end{theorem}

\noindent
The next theorem shows how to use approximation by sets with smooth boundaries to extend our results to sets with arbitrary boundaries.

\begin{theorem}[\cite{Hoermander1973}*{Thm.~2.6.9}]\label{thm:pseudoConvIntersection}
	Suppose that $K_i \subseteq \CC^n$, $i \in I$, are pseudo-convex sets for an index set $I$. Then the interior of $\bigcap_{i \in I}K_i$ is pseudo-convex.
\end{theorem}

\subsection{Brunn--Minkowski inequalities for complex moments}
For $K \in \convexbodies(\CC^n)$, $v \in \CC^n \setminus\{0\}$ and $p \geq 0$, the $p$-th asymmetric complex moment of $K$ is defined by
\begin{align*}
	\ComplPosReMom(K) = \int_{K \cap v^+} \Re(x\cdot v)^p dx, 
\end{align*}
where $v^+ = \{x \in \CC^n: \Re (x \cdot v) \geq 0 \}$. Note that clearly $\ComplPosReMom(K)$ is $(2n+p)$-homogeneous. A direct application of the Pr\'ekopa--Leindler inequality yields the following Brunn--Minkowski-type inequality for $\ComplPosReMom$ by Berck~\cite{Berck2009}.

\begin{proposition}\label{prop:BMComplMoments}
	Suppose that $p \geq 0$ and $v \in \CC^n \setminus\{0\}$. Then
	\begin{align*}
		\ComplPosReMom(K_0 + K_1)^{\frac{1}{2n+p}} \geq \ComplPosReMom(K_0)^{\frac{1}{2n+p}} + \ComplPosReMom(K_1)^{\frac{1}{2n+p}},
	\end{align*}
	for every $K_0, K_1 \in \convexbodies(\CC^n)$, such that $K_0 \cap v^+, K_1 \cap v^+ \neq \emptyset$.
\end{proposition}

\noindent
Like the classical Brunn--Minkowski-inequality, Proposition~\ref{prop:BMComplMoments} directly implies an analogue of Brunn's concavity theorem for the moments of parallel sections by complex hyperplanes $H_{u,z}=\{x \in \CC^n: x \cdot u = z\}$, that is, for
\begin{align*}
	\mathcal{M}_{p,v}^{\Re, +, u}(K,z) = \int_{v^+ \cap (K \cap H_{u,z})} \Re(x \cdot v)^p dx,
\end{align*}
where $u,v \in \CC^n\setminus\{0\}$ are not contained in the same complex line.

\begin{corollary}\label{cor:brunnComplMoment}
	Suppose that $K \in \convexbodies(\CC^n)$, $p>0$ and let $u, v \in \CC^n \setminus\{0\}$ with $v \not \in \linspan^\CC\{u\}$. Then the function 
	\begin{align*}
		z \mapsto \mathcal{M}_{p,v}^{\Re, +, u}(K,z)^{\frac{1}{2n-2+p}}
	\end{align*}
	is concave on the set $\{z \in \CC: K \cap H_{u,z} \cap v^+ \neq \emptyset\}$.
\end{corollary}
\begin{proof}
	This is a direct consequence of Proposition~\ref{prop:BMComplMoments} and the fact that \linebreak$(1-\lambda) (K \cap H_{u,z_0}) + \lambda (K \cap H_{u, z_1}) \subseteq K \cap H_{u, (1-\lambda)z_0 + \lambda z_1}$ by convexity.
\end{proof}

\noindent
As a concave function on its compact support, $\mathcal{M}_{p,v}^{\Re, +,u}(K,\cdot)^{\frac{1}{2n-2+p}}$ thus possesses a maximum, which is, by the monotonicity of $t \mapsto t^{2n-2+p}$, also true for $\mathcal{M}_{p,v}^{\Re, +,u}(K,\cdot)$. The following lemma shows that by replacing $v$ by $v+\lambda u$, with $\lambda \in \CC$ suitable, we can ensure that, for smooth $K$, $z=0$ is a critical point of $\mathcal{M}_{p,v}^{\Re, +,u}(K,\cdot)$ and thus its maximum. Note that for smooth $K$, $\mathcal{M}_{p,v+\lambda u}^{\Re, +,u}(K,\cdot)$ is differentiable at $z=0$, since the complex parallel section function of $K$ is smooth at $z=0$ (see, e.g., \cite{Koldobsky2005}*{Lem.~2.4}).

In the proof, we denote by $H^\RR_{u,t}$ the (real) hyperplane $\{x \in \CC^n: \langle x, u \rangle = t\}$, $u \in \CC^n \setminus\{0\}, t \in \RR$.

\begin{lemma}\label{lem:derZeroComplMoment}
	Suppose that $K \in \convexbodiesO(\CC^n)$ is $\unitsurf^1$-invariant and has smooth boundary, $p>0$ and let $u, v \in \CC^n \setminus\{0\}$ with $v \not \in \linspan^\CC\{u\}$. Then there exists $\lambda \in \CC$ such that
	\begin{align*}
		\left.\nabla_z \mathcal{M}_{p,v+\lambda u}^{\Re, +,u}(K,z)\right|_{z=0} = 0.
	\end{align*}
\end{lemma}
\begin{proof}
	Without loss of generality, we may assume that $v \cdot u = 0$. Letting \linebreak$K_0 = K \cap H^\RR_{iu,0}$, by \cite{Berck2009}*{Lem.~3.6} applied in $H^\RR_{iu,0}$, there exist $\lambda_1, \lambda_2 \in \RR$ such that
	\begin{align*}
		t \mapsto \int_{K_0 \cap H^\RR_{u,t} \cap (v+\lambda_1u)^+} \langle x, v + \lambda_1 u\rangle^p dx 
	\end{align*}
	and 
	\begin{align*}
		t \mapsto \int_{K_0 \cap H^\RR_{u,t} \cap (-iv+\lambda_2 u)^+} \langle x, -iv + \lambda_2 u\rangle^p dx 
	\end{align*}
	have critical points at $t=0$. Set $\lambda = \lambda_1 + i \lambda_2$. Next, since for $x \cdot u = t \in \RR$, we have $\Re(x \cdot(v+\lambda u)) = \Re(x \cdot (v + \lambda_1 u)) = \langle x, v + \lambda_1 u\rangle$ and $K \cap H_{u,t} = K_0 \cap H^\RR_{u,t}$,
	\begin{align*}
		\mathcal{M}_{p,v+\lambda u}^{\Re, +,u}(K,t) = \int_{K \cap H_{u,t} \cap (v+\lambda u)^+} \!\!\!\!\!\!\!\!\!\!\Re(x \cdot (v + \lambda u))^p dx = \int_{K_0 \cap H^\RR_{u,t} \cap (v+\lambda_1u)^+}\!\!\!\!\!\!\!\!\!\! \langle x, v + \lambda_1 u\rangle^p dx,
	\end{align*}
	and, hence, $t \mapsto \mathcal{M}_{p,v+\lambda u}^{\Re, +,u}(K,t)$ has a critical point at zero. If $x \cdot u = it$, $t \in \RR$, then $\Re(x \cdot(v+\lambda u)) = \langle x, v + \lambda_2 iu\rangle$ and, by the $\unitsurf^1$-invariance of $K$, $K \cap H_{u,it} = i(K_0 \cap H^\RR_{u,t})$. Consequently, by letting $x = iy$,
	\begin{align*}
		\mathcal{M}_{p,v+\lambda u}^{\Re, +,u}(K,t)&= \int_{K \cap H_{u,it} \cap (v+\lambda u)^+} \!\!\!\!\!\!\!\!\!\!\!\!\!\!\Re(x \cdot (v + \lambda u))^p dx= \int_{i(K_0 \cap H^\RR_{u,t} \cap (-iv+\lambda_2u)^+)}\!\!\!\!\!\!\!\!\!\!\!\!\!\! \langle x, v + \lambda_2 iu\rangle^p dx\\
		&=\int_{K_0 \cap H^\RR_{u,t} \cap (-iv + \lambda_2 u)^+}\!\!\!\!\!\!\!\!\!\! \langle y, -iv+ \lambda_2 u\rangle^p dy,
	\end{align*}
	we conclude that also $t \mapsto \mathcal{M}_{p,v+\lambda u}^{\Re, +,u}(K,it)$ has a critical point at zero, which yields the claim.
\end{proof}

\noindent
Using symmetries, Lemma~\ref{lem:derZeroComplMoment} now directly translates to symmetric moments,
\begin{align*}
	\mathcal{M}_{2,v}^{|\cdot|,u}(K,z) = \int_{K \cap H_{u,z}} |x \cdot v|^2 dx,
\end{align*}
whenever $K$ is $\unitsurf^1$-invariant.

\begin{proposition}\label{prop:SecSymmMomentMaxAtZero}
	Suppose that $K \in \convexbodiesO(\CC^n)$ is $\unitsurf^1$-invariant and let $u, v \in \CC^n \setminus\{0\}$ with $v \not \in \linspan^\CC\{u\}$. Then there exists $\lambda \in \CC$ such that
	\begin{align*}
		z \mapsto \mathcal{M}_{2,v+\lambda u}^{|\cdot|,u}(K,z)
	\end{align*}	
	 is maximal at $z=0$. 
\end{proposition}
\begin{proof}
	First note that since $K$ is $\unitsurf^1$-invariant,
	\begin{align*}
		\mathcal{M}_{2,w}^{|\cdot|,u}(K,z) = \mathcal{M}_{2,w}^{\Re, +,u}(K,z) + \mathcal{M}_{2,w}^{\Re, +,u}(K,-z) + \mathcal{M}_{2,w}^{\Re, +,u}(K,iz) + \mathcal{M}_{2,w}^{\Re, +,u}(K,-iz),
	\end{align*}
	for every $w \in \CC^n\setminus\{0\}$, and we need to choose $\lambda \in \CC$ such that $\mathcal{M}_{2,v +\lambda u}^{\Re, +,u}(K,\cdot)$ attains its maximum at $z=0$. However, by Corollary~\ref{cor:brunnComplMoment}, $\mathcal{M}_{2,v +\lambda u}^{\Re, +,u}(K,\cdot)^{\frac{1}{2n}}$ is concave on $\{z \in \CC: K \cap H_{u,z} \cap (v +\lambda u)^+ \neq \emptyset\}$, and by Lemma~\ref{lem:derZeroComplMoment} (together with the chain rule), there exists $\lambda \in \CC$ such that $\mathcal{M}_{2,v +\lambda u}^{\Re, +,u}(K,\cdot)^{\frac{1}{2n}}$ is maximal at $z=0$ for smooth $K \in \convexbodiesO(\CC^n)$. Hence, the claim follows from the monotonicity of $t \mapsto t^{2n}$ and by approximating a general $K \in \convexbodiesO(\CC^n)$ by smooth bodies.
\end{proof}

\subsection{Proof of Theorem~\ref{mthm:PlushIntBody}}

In this section, we compute the necessary derivatives required in order to apply Theorem~\ref{thm:leviCondPseudoConvex} in the proof of Theorem~\ref{mthm:PlushIntBody}. 

\medskip

Recalling that the analytic family of distributions $r^q_+$, $\Re q > -1$, can be extended analytically to $-\Re q \not \in \NN_+$, and that, for $-2<\Re q < -1$, this extension is given by
\begin{align}\label{eq:DistrRqAnExt}
	\langle r^q_+, \Phi\rangle = \int_0^\infty r^{q} \left( \Phi(r) - \Phi(0) - r\Phi'(0)\right) dr, \quad \Phi \in C_c^\infty(\CC),
\end{align}
which clearly can be extended to all $\Phi \in C(\CC)$ with compact support, which are smooth in a neighborhood of zero. The main auxiliary result can then be stated as follows.

\begin{proposition}\label{prop:DerRadFctIntBodyD}
	Suppose that $p>-2$, $p \neq 0, -1$, and let $u, w \in \CC^n \setminus\{0\}$ with $w \not \in \linspan^\CC\{u\}$. Then
	\begin{align}\label{eq:derivRadFctSph}
			\Delta_z (\rho_{\IntBody_{\DD, p}K}(u+zw)^{-p})|_{z=0} = 2\pi p^2 \langle r^{p-1}_+,\mathcal{M}_{2,w}^{|\cdot|,u}(K,\cdot)\rangle
		\end{align}
	for every $\unitsurf^1$-invariant $K \in \convexbodiesO(\CC^n)$ with smooth boundary.
\end{proposition}
\begin{proof}
	Assume first that $p>0$. Since $z \mapsto \rho_{\IntBody_{\DD, p}K}(u+zw)^{-p}$ is a tempered distribution on $\CC$, we can consider the Fourier transform (denoted by $\hat\cdot$) of its Laplacian, applied to a Schwartz function $\varphi$ on $\CC$, that is,
	\begin{align*}
		\langle (\Delta_z \rho_{\IntBody_{\DD, p}K}(u+zw)^{-p})^{\widehat{ }}, \varphi \rangle\! =\! \langle \rho_{\IntBody_{\DD, p}K}(u+zw)^{-p}, \Delta_z \widehat{\varphi}\rangle \!=\! \langle \rho_{\IntBody_{\DD, p}K}(u+zw)^{-p},- \widehat{(|\cdot|^2 \varphi)}\rangle.
	\end{align*}
	By inserting the definition of $\rho_{\IntBody_{\DD, p} K}$, exchanging the order of integration, and letting $c = z - \overline{(x \cdot u)/(x \cdot w)}$
	\begin{align*}
		\langle \rho_{\IntBody_{\DD, p}K}(u+zw)^{-p}, \widehat{|\cdot|^2 \varphi}\rangle &= \int_{K} \int_\CC |x \cdot (u+zw)|^p \widehat{(|\cdot|^2 \varphi)}(z) dz dx\\
		&=\int_{K} \int_\CC |x \cdot w|^p |z + \overline{(x \cdot u)/(x \cdot w)}|^p \widehat{(|\cdot|^2 \varphi)}(z) dz dx\\
		&=\int_{K}  |x \cdot w|^p \int_\CC |c|^p \widehat{(|\cdot|^2 \varphi)}(c - \overline{(x \cdot u)/(x \cdot w)}) dc dx.
	\end{align*}
	Next, it is a direct computation that for $p\neq -2, -4, \dots$ (see, e.g., \cite{Gelfand1964}*{Sec.~II.3.3}), 
	\begin{align*}
		|\cdot|^2 \widehat{(|\cdot|^p)} = -p^2\widehat{(|\cdot|^{p-2})},
	\end{align*}
	and, consequently, the previous integral simplifies to
	\begin{align*}
		-p^2\int_{K}  |x \cdot w|^p \int_\CC |c|^{p-2} \widehat{\varphi}(c - \overline{(x \cdot u)/(x \cdot w)}) dc dx \\
		= -p^2\int_\CC \int_{K} |x \cdot w|^2  |x \cdot(u + zw)|^{p-2}dx\, \widehat{\varphi}(z) dz.
	\end{align*}
	By taking the inverse Fourier transform, we conclude that
	\begin{align}\label{eq:prfDerRadFctIntBodyDistrEqu}
		\Delta_z \rho_{\IntBody_{\DD, p}K}(u+zw)^{-p} = p^2 \int_{K} |x \cdot w|^2  |x \cdot(u + zw)|^{p-2}dx
	\end{align}
	as tempered distributions. By
	\begin{align*}
		\langle \Delta_z \rho_{\IntBody_{\DD, p}K}(u+zw)^{-p}, \varphi \rangle = \langle \rho_{\IntBody_{\DD, p}K}(u+zw)^{-p}, \Delta_z\varphi \rangle,
	\end{align*}
	$\varphi \in C_c^\infty(\CC)$, and since $\rho_{\IntBody_{\DD, p}K}(u+zw)^{-p}$ is analytic in $p$, the left-hand side of \eqref{eq:prfDerRadFctIntBodyDistrEqu} is an analytic family of distributions (in $z \in \CC$). Rewriting the right-hand side of \eqref{eq:prfDerRadFctIntBodyDistrEqu} by Fubini's theorem, and using polar coordinates on $\CC$ and the $\unitsurf^1$-invariance of $K$,
	\begin{align*}
		p^2 \int_{K} |x \cdot w|^2  |x \cdot(u + zw)|^{p-2}dx &= p^2 \int_\CC |\zeta|^{p-2} \int_{K \cap H_{u+zw,\zeta}} |x \cdot w|^2 dx d\zeta\\
		&= p^2 \int_0^\infty r^{p-1} \int_{\unitsurf^1}\int_{K \cap H_{u+zw,rc}} |x \cdot w|^2 dx dc dr \\
		&= 2\pi p^2 \int_0^\infty r^{p-1} \int_{K \cap H_{u+zw,r}} |x \cdot w|^2 dx dr\\
		&= 2\pi p^2 \langle r_+^{p-1}, \mathcal{M}_{2,w}^{|\cdot|,u+zw}(K,r) \rangle_r,
	\end{align*}
	we conclude that also the right-hand side of \eqref{eq:prfDerRadFctIntBodyDistrEqu} is an analytic family of distributions. The uniqueness of analytic continuation therefore implies that
	\begin{align*}
		\langle \Delta_z \rho_{\IntBody_{\DD, p}K}(u+zw)^{-p}, \varphi \rangle = 2\pi p^2 \langle \langle r_+^{p-1}, \mathcal{M}_{2,w}^{|\cdot|,u+zw}(K,r) \rangle_r, \varphi \rangle_z,
	\end{align*}
	for all $\varphi \in C_c^\infty(\CC)$ and $p > -2$, $p \neq 0, -1$. Note that since (for $p>0$)
	\begin{align*}
		\langle \langle r_+^{p-1}, \mathcal{M}_{2,w}^{|\cdot|,u+zw}(K,r) \rangle_r, \varphi \rangle_z &= \int_\CC \langle r_+^{p-1}, \mathcal{M}_{2,w}^{|\cdot|,u+zw}(K,r) \rangle_r \varphi (z) dz \\
		&=\langle r_+^{p-1}, \langle\mathcal{M}_{2,w}^{|\cdot|,u+zw}(K,r),\varphi \rangle_z \rangle_r
	\end{align*}
	the analytic continuation of $\langle r_+^{p-1}, \mathcal{M}_{2,w}^{|\cdot|,u+zw}(K,r) \rangle_r$ is given by
	\begin{align*}
		\langle \langle r_+^{p-1}, \mathcal{M}_{2,w}^{|\cdot|,u+zw}(K,r) \rangle_r, \varphi \rangle_z = \int_0^\infty r_+^{p-1} \left( \langle\mathcal{M}_{2,w}^{|\cdot|,u+zw}(K,r),\varphi \rangle_z \right. \\ \left.- \langle\mathcal{M}_{2,w}^{|\cdot|,u+zw}(K,0),\varphi \rangle_z  - \langle r \left.\frac{\partial}{\partial r}\right|_{r=0}\mathcal{M}_{2,w}^{|\cdot|,u+zw}(K,r),\varphi \rangle_z \right) dr\\
		= \langle \langle r_+^{p-1},\mathcal{M}_{2,w}^{|\cdot|,u+zw}(K,r)\rangle_r, \varphi \rangle_z .
	\end{align*}
	Since $K=-K$, $\mathcal{M}_{2,w}^{|\cdot|,u+zw}(K,r)$ is even (in $r$), the derivative at $r=0$ vanishes. Consequently,
	\begin{align}\label{eq:prfDerRadFctIntBodyDistrEqu2}
		\Delta_z \rho_{\IntBody_{\DD, p}K}(u+zw)^{-p} &= 2\pi p^2  \langle r_+^{p-1},\mathcal{M}_{2,w}^{|\cdot|,u+zw}(K,r)\rangle_r \\
		&= 2\pi p^2\int_0^\infty r^{p-1} \left(\mathcal{M}_{2,w}^{|\cdot|,u+zw}(K,r) - \mathcal{M}_{2,w}^{|\cdot|,u+zw}(K,0)\right) dr,\nonumber
	\end{align}
	as distributions. Next, observe that since the operator $\JOp_{\DD, p}$ commutes with the action of $\mathrm{U}(n)$ on $\unitsurf^{2n-1}$, $\JOp_{\DD, p}$ maps $C^\infty(\unitsurf^{2n-1})$ to itself. Consequently, by \eqref{eq:IntBodyByJOp} and as $\rho_{\IntBody_{\DD,p} K}(x)$ is strictly positive for $x \neq 0$, $\rho_{\IntBody_{\DD,p} K}$ is smooth in $\CC^n \setminus\{0\}$, whenever $K \in \convexbodiesO(\CC^n)$ has a smooth boundary.
	
	As the right-hand side of \eqref{eq:prfDerRadFctIntBodyDistrEqu2} is also continuous in $z$, both sides of \eqref{eq:prfDerRadFctIntBodyDistrEqu2} coincide as functions (as $u+zw \neq 0$ for all $z \in \CC$, by assumption). Evaluating at $z=0$, we obtain
	\begin{align*}
		\Delta_z \rho_{\IntBody_{\DD, p}K}(u+zw)^{-p}|_{z=0} = 2 \pi p^2\langle r_+^{p-1}, \mathcal{M}_{2,w}^{|\cdot|,u}(K,\cdot)\rangle,
	\end{align*}
	which yields the claim.
\end{proof}

\noindent
The last ingredient of the proof of Theorem~\ref{mthm:PlushIntBody} is the following result from elementary calculus, included for the reader's convenience.

\begin{lemma}\label{lem:LaplVersch}
	Suppose that $F \in C^\infty(\CC^n\setminus\{0\})$ is non-negative and $F(u)>0$ for $u \neq 0$, one-homogeneous and $\unitsurf^1$-invariant, that is, $F(zw) = |z|F(w)$, $z \in \CC, w \in \CC^n$, and let $u, v \in \CC^n \setminus\{0\}$ with $\nabla F(u) \cdot v = 0$. Then
	\begin{align*}
		\left.\Delta\right|_{z=0} F(u + zw)^p  = p^2 |\lambda|^2 F(u)^p + p F(u)^{p-1} \left.\Delta\right|_{z=0} F(u + zv),
	\end{align*}
	where $w = v + \lambda u$, $\lambda \in \CC$, and the derivatives are with respect to $z \in \CC$.
\end{lemma}
\begin{proof}
	First note that by one-homogeneity and $\unitsurf^1$-invariance,
	\begin{align}\label{eq:prfLemDerivHomFct}
		\langle \nabla F(u), u\rangle = F(u) \quad \text{ and } \quad \langle \nabla F(u), iu \rangle = 0,
	\end{align}
	and, by differentiating the equalities in \eqref{eq:prfLemDerivHomFct},
	\begin{align}\label{eq:prfLemDerivHomFctSec}
		d^2F(u)u = 0 \quad \text{ and } \quad d^2 F(u) iu = i \nabla F(u).
	\end{align}
	Next, computing by the chain rule, for $x \in \CC^n \setminus \{0\}$ arbitrary, yields
	\begin{align*}
		\ddt\!\!\!\!\! F(u+tx)^p = p(p-1)F(u)^{p-2} \langle \nabla F(u),x\rangle^2 + pF(u)^{p-1} \langle x, d^2 F(u) x\rangle.
	\end{align*}
	Letting $x = v+\lambda u$ and applying \eqref{eq:prfLemDerivHomFct}, \eqref{eq:prfLemDerivHomFctSec} and the assumptions on $v$,
	\begin{align*}
		\ddt\!\!\!\!\! F(u+tw)^p = p(p-1)F(u)^{p} (\Re \lambda)^2 + pF(u)^{p-1} \langle v, d^2 F(u) v\rangle + pF(u)^{p} (\Im \lambda)^2,
	\end{align*}
	and for $x = i(v+\lambda u)$,
	\begin{align*}
		\ddt\!\!\!\!\! F(u+tiw)^p = p(p-1)F(u)^{p} (\Im \lambda)^2 + pF(u)^{p-1} \langle iv, d^2 F(u) iv\rangle + pF(u)^{p} (\Re \lambda)^2,
	\end{align*}
	which yields the claim, when summed up.
\end{proof}

\noindent
We are now ready to prove Theorem~\ref{mthm:PlushIntBody}.

\begin{proof}[Proof of Theorem~\ref{mthm:PlushIntBody}]
	By~\eqref{eq:S1invariantC=D1} and~\eqref{eq:S1invariantC=D2}, we can assume without loss of generality that $C = \DD$. Moreover, by Theorem~\ref{prop:compLp=realLp} (which is proved independently in Section~\ref{sec:ProofIntIneq}) and Theorem~\ref{thm:berckConvexityLpIntBody}, we only need to consider $-2<p<-1$. 
	
	Let now $K \in \convexbodiesO(\CC^n)$ be $\unitsurf^1$-invariant and assume first that its radial function $\rho_K$ is smooth in $\CC^n \setminus \{0\}$. Noting, as before, that $\rho_{\IntBody_{\DD,p} K}$ is smooth in $\CC^n \setminus\{0\}$, and
	\begin{align*}
		\interior \IntBody_{\DD,p} K = \{u \in \CC^n: \rho_{\IntBody_{\DD,p} K}(u)^{-1} - 1 < 0\},
	\end{align*}
	by Theorem~\ref{thm:leviCondPseudoConvex}, we need to show that 
	\begin{align*}
		\Delta_z \left(\rho_{\IntBody_{\DD,p} K}(u+zv)^{-1}\right)|_{z=0} \geq 0
	\end{align*}
	for all $u \in \mathrm{bd}\, \IntBody_{\DD,p} K$ and $\nabla (\rho_{\IntBody_{\DD,p} K}^{-1})(u) \cdot v = 0$.
	
	Therefore, let $u \in \mathrm{bd}\, \IntBody_{\DD,p} K$ be fixed and take $v \in \CC^n \setminus \{0\}$ arbitrary such that $\nabla (\rho_{\IntBody_{\DD,p} K}^{-1}) \cdot v = 0$. If $v = \zeta u$ for some $\zeta \in \CC$, then, since $\rho_{\IntBody_{\DD,p} K}(u)=1$, the \linebreak$\unitsurf^1$-invariance and homogeneity of the radial function imply that
	\begin{align*}
		\rho_{\IntBody_{\DD,p} K}(u+zv)^{-1} = |1+z\zeta|,
	\end{align*}
	and one directly sees that $\Delta_z |1+z\zeta| \geq 0$ at $z=0$. If $v \not \in \linspan^\CC\{u\}$, by Lemma~\ref{lem:LaplVersch},
	\begin{align*}
		\Delta_z \left(\rho_{\IntBody_{\DD,p} K}(u+zv)^{-1}\right)|_{z=0} + p|\lambda|^2 = \frac{1}{p}\Delta_z \left(\rho_{\IntBody_{\DD,p} K}(u+zw)^{-p}\right)|_{z=0},
	\end{align*}
	with $w = v + \lambda u$, for some $\lambda \in \CC$ to be chosen later, which, by Proposition~\ref{prop:DerRadFctIntBodyD} is equal to
	\begin{align*}
		2\pi p \langle r^{p-1}_+, \mathcal{M}_{2,w}^{|\cdot|,u}(K,\cdot)\rangle. 
	\end{align*}
	Denoting $\Phi(r) = \mathcal{M}_{2,w}^{|\cdot|,u}(K,r)$, we conclude from \eqref{eq:DistrRqAnExt} that
	\begin{align*}
		\Delta_z \left(\rho_{\IntBody_{\DD,p} K}(u+zv)^{-1}\right)|_{z=0} + p|\lambda|^2 = 2 \pi p \int_{0}^\infty r^{p-1} (\Phi(r) - \Phi(0) - r\Phi'(0)) dr.
	\end{align*}
	Since $K$ is origin-symmetric, $\Phi$ is even, and, hence, $\Phi'(0) = 0$. Moreover, by Proposition~\ref{prop:SecSymmMomentMaxAtZero}, we can choose $\lambda \in \CC$ such that $\Phi(r) \leq \Phi(0)$ for all $r>0$. As $p<0$, we conclude that
	\begin{align*}
		\Delta_z \left(\rho_{\IntBody_{\DD,p} K}(u+zv)^{-1}\right)|_{z=0} \geq 0,
	\end{align*}
	that is, $\interior \IntBody_{\DD,p} K$ is pseudo-convex.
	
	For general $\unitsurf^1$-invariant $K \in \convexbodiesO(\CC^n)$, we approximate $K$ by smooth $\unitsurf^1$-invariant convex bodies $K_j$, $j \in \NN$ such that $K \subseteq K_j$ for all $j \in \NN$. By the first part of the proof and the monotonicity of $\IntBody_{\DD,p}$,
	\begin{align*}
		\IntBody_{\DD, p} K = \bigcap_{j \in \NN} \IntBody_{\DD, p} K_j,
	\end{align*}
	where all $\interior \IntBody_{\DD, p} K_j$ are pseudo-convex, and the claim follows by Theorem~\ref{thm:pseudoConvIntersection}.
\end{proof}

\subsection{Counterexamples to convexity}\label{sec:counterExConv}
In the proof of Theorem~\ref{mthm:PlushIntBody}, $\unitsurf^1$-invariance of the convex bodies played a critical role. It is therefore a natural question to ask whether this is a particular aspect of the proof or reflects an underlying principle. In this section, we give a (partial) answer to this by providing examples in the range $-1\leq p<1$ of convex bodies which are not $\unitsurf^1$-invariant and have non-convex, complex $L_p$-intersection bodies.

These examples are obtained by considering sequences of ellipsoids, whose complex $L_p$-intersection bodies converge to a non-convex star body. The key ingredient of this argument is the following generalization of (parts of) \cite{Grinberg1999}*{Lem.~6.3}, proved using similar arguments.

\begin{lemma}\label{lem:seqEllDeltaNewContinuous}
	Suppose that $p \geq -1$ and let $\bar e \in \unitsurf^{2n-1}$. Then there exists a sequence of origin-symmetric ellipsoids $E_j \subseteq \CC^n, j \in \NN$, such that
	\begin{align}\label{eq:seqEllDeltaNew}
		\lim_{j \to \infty}\int_{\unitsurf^{2n-1}} f(u) \rho_{E_j}(u)^{2n+p} du = \frac{1}{2}(f(\bar e) + f(-\bar e)),
	\end{align}
	for all $f \in C(\unitsurf^{2n-1})$.
\end{lemma}	
\begin{proof}
	First, without loss of generality, we may assume that $\bar e$ is the first standard unit vector in $\CC^n$. Using generalized spherical coordinates $u = (u_1 \sin(t), u_2 \cos(t))$ for $u \in \unitsurf^{2n-1}$, with $u_1 \in \unitsurf^0 = \{\pm \bar e\}, u_2 \in \unitsurf^{2n-2}$ and $t \in [0,\pi/2]$, the radial function of the ellipsoid
	\begin{align*}
		E_{a,b} = \left\{(z_1, \dots, z_{n}) \in \CC^n: \frac{(\Re z_1)^2}{a^2} + \frac{(\Im z_1)^2 + |z_2|^2 + \dots + |z_n|^2}{b^2} \leq 1\right\},
	\end{align*}
	for $a,b>0$ is given by
	\begin{align*}
		\rho_{E_{a,b}}(u_1 \sin(t), u_2 \cos(t)) = \left(\frac{\sin(t)^2}{a^2} + \frac{\cos(t)^2}{b^2} \right)^{-1/2}, \quad t \in [0,\pi/2].
	\end{align*}
	Next, choose $b_j>0$ by the intermediate value theorem, such that
	\begin{align}\label{eq:prfSeqEllDeltaMass1}
		\int_{\unitsurf^{2n-1}}\rho_{E_{j, b_j}}(u)^{2n+p} du = 1,
	\end{align}
	whenever $j \in \NN$ is large enough. Note that $b_j \to 0$ as $j \to \infty$. Indeed, assume that $b_j \geq M$ for some constant $M>0$. Writing \eqref{eq:prfSeqEllDeltaMass1} in generalized spherical coordinates (where $du = \cos(t)^{2n-2}du_1 du_2 dt$), denoting $C_n = 2 (2n-2)\kappa_{2n-2}$,
	\begin{align*}
		1 = \int_{\unitsurf^{2n-1}}\!\!\!\!\!\rho_{E_{j,b_j}}(u)^{2n+p} du &= C_n \int_0^{\pi/2}\!\!\!\! \cos(t)^{2n-2} \left(\frac{\sin(t)^2}{j^2} + \frac{\cos(t)^2}{b_j^2} \right)^{-(2n+p)/2}\!\!\!\!\! dt\\
		&\geq C_n\int_0^{\pi/2}\!\!\!\! \cos(t)^{2n-2} \left(\frac{\sin(t)^2}{j^2} + \frac{\cos(t)^2}{M^2} \right)^{-(2n+p)/2}\!\!\!\!\! dt,
	\end{align*}
	and letting $j \to \infty$ yields (by monotone convergence)
	\begin{align*}
		1 \geq C_n M^{2n+p} \int_0^{\pi/2}\cos(t)^{-2-p} dt,
	\end{align*}
	which contradicts the fact that $\cos(t)^{-2-p}$ is integrable only if $-2-p>-1$, that is $p < -1$. Since $b_j$ is clearly monotonously decreasing, $b_j \to 0$.
	
	Setting $E_j = E_{j,b_j}$, we claim that a subsequence of $(E_j)_{j\in \NN}$ already satisfies \eqref{eq:seqEllDeltaNew}. Indeed, observe that by \eqref{eq:prfSeqEllDeltaMass1} and since they are positive, the functions $\rho_{E_{j}}^{2n+p}$ all have norm $1$, when seen as elements of the dual space of $C(\unitsurf^{2n-1})$. Consequently, by the Banach--Alaoglu theorem, there exists a subsequence (again denoted by $(\rho_{E_{j}}^{2n+p})_j$) converging in the weak-* topology to a Borel measure $\mu$ on $\unitsurf^{2n-1}$, that is,
	\begin{align*}
		\int_{\unitsurf^{2n-1}} f(u) \rho_{E_{j}}(u)^{2n+p} du \to \int_{\unitsurf^{2n-1}} f(u) d\mu(u), \quad j \to \infty,
	\end{align*}
	for every $f \in C(\unitsurf^{2n-1})$. Showing $\mu = \frac{1}{2}(\delta_{\bar e} + \delta_{-\bar e})$ thus directly implies the claim.
	
	To this end, suppose that $u \in \unitsurf^{2n-1} \setminus\{\pm \bar e\}$ and let $U \subseteq \unitsurf^{2n-1}$ be an open neighborhood of $u$ not containing $\pm \bar e$ in its closure. Then there exists $\varepsilon > 0$ such that for all $u=(u_1\sin(t), u_2\cos(t)) \in U$, we have $\cos(t) > \varepsilon$, and, therefore
	\begin{align}\label{eq:prfSeqEllDeltaNewCtEstimOutside}
		\left(\frac{\sin(t)^2}{j^2} + \frac{\cos(t)^2}{b_j^2} \right)^{-(2n+p)/2} \leq \left(\frac{\sin(t)^2}{j^2} + \frac{\varepsilon^2}{b_j^2} \right)^{-(2n+p)/2} \leq \frac{b_j^{2n+p}}{\varepsilon^{2n+p}}.
	\end{align}
	As $b_j \to 0$ for $j \to \infty$, a direct estimate for $f \in C(\unitsurf^{2n-1})$ concentrated on $U$, shows
	\begin{align*}
		\int_{\unitsurf^{2n-1}} f(c) d\mu(c) = \lim_{j \to \infty} \int_{U} f(c) \rho_{E_{j}}(c)^{2n+p} dc = 0,
	\end{align*}
	that is, $\supp \mu \subseteq \unitsurf^{2n-1} \setminus U$ and, hence, $\supp \mu \subseteq \{\pm \bar e\}$, as $u$ was arbitrary. Since $\mu(\unitsurf^{2n-1})=1$, by \eqref{eq:prfSeqEllDeltaMass1}, and $\mu$ must be even (as weak-* limit of even measures), we conclude that $\mu = \frac{1}{2}(\delta_{\bar e} + \delta_{-\bar e})$, which completes the proof. 
\end{proof}

\noindent
The previous lemma for continuous functions on $\unitsurf^{2n-1}$ can be directly extended to functions with a specific type of pole.

\begin{lemma}\label{lem:seqEllDeltaRhoC}
	Suppose that $p \geq -1$ and let $\bar e \in \unitsurf^{2n-1}$. Then there exists a sequence of origin-symmetric ellipsoids $E_j \subseteq \CC^n, j \in \NN$, such that
	\begin{align}\label{eq:seqEllDeltaNewRadFct}
		\lim_{j \to \infty}\int_{\unitsurf^{2n-1}} h_C(v \cdot (c \bar e))^p \rho_{E_j}(u)^{2n+p} dv = \frac{1}{2}\left(h_C(\overline{c})^p + h_C(-\overline{c})^p\right), \quad c \in \unitsurf^1,
	\end{align}
	for all $C \in \convexbodiesO(\CC)$.
\end{lemma}
\begin{proof}
	If $p>0$, the function $g(v) = h_C(v\cdot (c\bar e))^p$ is continuous on $\unitsurf^{2n-1}$ and, the claim follows from Lemma~\ref{lem:seqEllDeltaNewContinuous}. For $p<0$, set $M = 2\max\{h_C(\overline{c})^p, h_C(-\overline{c})^p\}$, and consider the decomposition
	\begin{align*}
		g(v) = \min\{g(v), M\} + ( \max\{g(v), M\} - M),
	\end{align*}
	where the first function clearly is continuous and coincides with $g$ on a neighborhood $V$ of $\pm \bar e$, whereas the second function vanishes on the same neighborhood. Taking the ellipsoids $E_j$ as in the previous lemma, the same estimate as in \eqref{eq:prfSeqEllDeltaNewCtEstimOutside} implies that
	\begin{align*}
		\int_{\unitsurf^{2n-1}} (\max\{g(v), M\} - M) \rho_{E_j}(u)^{2n+p} dv \leq \frac{b_j^{2n+p}}{\varepsilon^{2n+p}} \int_{\unitsurf^{2n-1} \setminus V}\!\! (\max\{g(v), M\} - M) dv,
	\end{align*}
	where the integral on the right-hand side is finite, since its absolute value is bounded by $\|\JOp_{C, p} 1\|_{\infty} + M(2n-1)\kappa_{2n-1}$. Consequently, as $b_j \to 0$, the left-hand side converges to zero as $j \to \infty$. Hence, together with Lemma~\ref{lem:seqEllDeltaNewContinuous} for the first term $\min\{g(v),M\}$ and since $\min\{g(v),M\} = g(v)$ for $v = \pm \bar e$, the claim follows.
\end{proof}

\noindent
We are now ready to state the aforementioned counterexample.

\begin{proposition}
	Let $-1\leq p<1$ be non-zero. Then there exists $C \in \convexbodiesO(\CC)$ and an origin-symmetric ellipsoid $K \subseteq \CC^n$ such that $\IntBody_{C,p} K$ is not convex.
\end{proposition}
\begin{proof}
	By \eqref{eq:IntBodyByJOp} and Lemma~\ref{lem:seqEllDeltaRhoC}, there exists $\bar{e} \in \unitsurf^{2n-1}$ and a sequence $(E_j)_{j \in \NN}$ of origin-symmetric ellipsoids such that
	\begin{align*}
		\rho_{\IntBody_{C,p} E_j}(c \bar e)^{-p} \to \frac{1}{2(2n+p)}(h_C(\overline{c})^p + h_C(-\overline{c})^p), \quad j \to \infty,
	\end{align*}
	for every $c \in \unitsurf^1$, $C \in \convexbodiesO(\CC)$ and $p \geq -1$.
	Note that, when choosing $C$ to be, e.g., a suitable triangle, the function $c \mapsto (h_C(c)^p + h_{-C}(c)^p)^{-1/p}$ is not the radial function of a convex body (see \cite{Gardner2006}*{Sec.~6.1} for details) when $p<1$. Consequently, the radial function of $\IntBody_{C,p} E_j$ converges pointwise to the radial function of a non-convex star body as $j \to \infty$ and, hence, $\IntBody_{C,p} E_j$ cannot be convex when $j$ is sufficiently large.
\end{proof}


\section{Proof of Theorems~\ref{prop:compLp=realLp} and \ref{mthm:ineqComplexLpVsLp}}\label{sec:ProofIntIneq}
In this section we establish a representation of the radial function of $\IntBody_{C,p} K$ for $K \in \starbodiesO(\CC^n)$ and origin-symmetric $C \in \convexbodiesO(\CC)$, and use it to prove Theorem~\ref{mthm:ineqComplexLpVsLp}.

\medskip

To this end, let $L \in \convexbodiesO(\RR^d)$ be origin-symmetric, $\|\cdot\|_L = \rho_L^{-1}$ its gauge function, and recall that the space $(\RR^d, \|\cdot\|_L)$ \emph{embeds (isometrically) in $L_q$}, $q>0$, that is, in $L_q([0,1])$, if and only if there exists a finite Borel measure $\mu$ on $\unitsurf^{d-1}$, such that
\begin{align}\label{eq:embedIsomLpPos}
 \|x\|_L^q = \int_{\unitsurf^{d-1}} |\langle x, u\rangle|^q d\mu(u), \quad x \in \RR^d,
\end{align}
see, e.g., \cite{Koldobsky2005}*{Ch.~6} for details. Interpreting \eqref{eq:embedIsomLpPos} as an equality of distributions and using the relation between Radon and Fourier transform $\hat{\cdot}$, this definition can be formally extended to negative values of $p$ as follows (\cite{Koldobsky2005}*{Def.~6.14}).
\begin{definition}\label{def:embedIsomLpNeg}
 Suppose that $L \in \starbodiesO(\RR^d)$ is origin-symmetric. Then the space $(\RR^d, \|\cdot\|_L)$ is said to \emph{embed} in $L_{-q}, 0<q<n$, if there exists a finite Borel measure~$\mu$ on $\unitsurf^{d-1}$ such that
 \begin{align}\label{eq:embedIsomLpNeg}
  \int_{\RR^d} \|x\|_L^{-q} \phi(x) dx = \int_{\unitsurf^{d-1}} \left( \int_0^\infty t^{q-1} \hat \phi(tu) dt \right) d\mu(u),
 \end{align}
 for every even Schwartz function $\phi$ on $\RR^d$.
\end{definition}
Note that this definition is closely related to the notion of $k$-intersection bodies. More precisely, an origin-symmetric star body $L \in \starbodiesO(\RR^d)$ is a $k$-intersection body if and only if the space $(\RR^d, \|\cdot\|_L)$ embeds in $L_{-k}$.

The following lemma is a direct consequence of \cite{Koldobsky2005}*{Cor.~6.7 \& 6.8}, for $0<q\leq 1$, and of \cite{Koldobsky2005}*{Thm.~6.17}, for $-2<q<0$.
\begin{lemma}\label{lem:2DimBodyEmbedLp}
 For every origin-symmetric $L \in \convexbodiesO(\RR^2)$, the space $(\RR^2,\|\cdot\|_L)$ embeds in $L_q$ for every non-zero $-2<q \leq 1$.
\end{lemma}


\noindent
We further require a lemma relating the Fourier transforms of the complex and the real Radon transform.

\begin{lemma}\label{lem:ComplRealRadonFourier}
 Suppose that $f \in C(\CC^n)$ has compact support and, for $u \in \CC^n\setminus\{0\}$, recall that for $z \in \CC$ and $t \in \RR$,
 \begin{align*}
  (\mathcal{R}_u^\CC f)(z) = \int_{x \cdot u = z} f(x) dx \quad \text{ and } \quad (\mathcal{R}_u^\RR f)(t) = \int_{\langle x,u\rangle = t} f(x) dx,
 \end{align*}
 denote the \emph{complex and real Radon transforms} of $f$.
 Then
 \begin{align*}
  \widehat{\mathcal{R}_u^\CC f}(rc) = \widehat{\mathcal{R}_{cu}^\RR f}(r), \quad r \in \RR, c \in \CC,
 \end{align*}
 for $u \in \CC^n\setminus\{0\}$, where the left Fourier transform is on $\CC$ and the right one on $\RR$.

\end{lemma}
\begin{proof}
The claim follows by Fubini's theorem applied twice and \eqref{eq:RelSkalProdCReal},
 \begin{align*}
  \widehat{\mathcal{R}_u^\CC f}(rc) &= \int_{\CC}\int_{x \cdot u = z} f(x) e^{-i\langle rc,z\rangle} dx dz = \int_{\CC^n}f(x) e^{-i\langle rc,x \cdot u\rangle} dx \\
  &=\int_{\CC^n}f(x) e^{-ir\langle cu,x \rangle} dx = \int_{\RR}\int_{\langle x, cu\rangle = t} f(x) e^{-irt} dx dt = \widehat{\mathcal{R}_{cu}^\RR f}(r).
 \end{align*}
\end{proof}

\bigskip

\noindent We are now in a position to prove the main proposition required in the proof of Theorem~\ref{prop:compLp=realLp}. In the statement of the proposition, we denote by
\begin{align*}
	K^\circ = \{z \in \CC^n: \langle z, y\rangle \leq 1, \, \forall y \in K\}
\end{align*}
the polar body of $K \in \convexbodiesO(\CC^n)$.

\begin{proposition}\label{prop:repCIntBodyByLpIntBody}
Let $C \in \convexbodiesO(\CC)$ be origin-symmetric and $-1 \leq p < 1$ be non-zero. Then there exists a finite Borel measure $\mu_{C,p}$ on $\unitsurf^1$, such that
\begin{align}\label{eq:proprepCIntBodyByLpIntBody}
\rho_{\IntBody_{C,p}K}(u)^{-p}=\int_{\unitsurf^{1}}\rho_{\IntBody_p K}(cu)^{-p} d\mu_{C,p}(c), \quad u \in \unitsurf^{2n-1},
\end{align}
for every $K \in \starbodiesO(\CC^n)$, where, for $p=-1$, $d\mu_{C,-1}=\frac 12\rho_{iC^\circ}(c) dc$. In particular, if $K \in \starbodiesO(\CC^n)$ is $\unitsurf^1$-invariant, then $\IntBody_{C,p}K =\mu_{C,p}(\unitsurf^1)\IntBody_p K$.
\end{proposition}
\begin{proof}
We distinguish the cases $0<p<1$, $-1<p<0$ and $p=-1$. If $p>0$, by Lemma~\ref{lem:2DimBodyEmbedLp} applied to $C^\circ \in \convexbodiesO(\CC)$, there exists a finite Borel measure $\mu_{C,p}$ on $\unitsurf^1$ such that 
\begin{align}\label{eq:prfPropRepCIntBodyPosPDarstC}
 h_C(z)^p = \rho_{C^\circ}(z)^{-p} = \|z\|_{C^\circ}^p = \int_{\unitsurf^1} |\langle z, c\rangle|^p d\mu_{C,p}(c), \quad z \in \CC.
\end{align}
Note that we identify $\CC \cong \RR^2$ here. Combining, for $u \in \unitsurf^{2n-1}$, \eqref{eq:prfPropRepCIntBodyPosPDarstC} with the definition~\eqref{eq:defCompLpIntersectBody} of $\IntBody_{C,p} K$, $K \in \starbodiesO(\CC^n)$, by \eqref{eq:RelSkalProdCReal}, and interchanging the order of integration,
\begin{align*}
 \rho_{\IntBody_{C,p} K}(u)^{-p} &= \int_K h_C(x \cdot u)^p dx = \int_K \int_{\unitsurf^1} |\langle x \cdot u, c \rangle|^p d\mu_{C,p}(c) dx \\
  &= \int_{\unitsurf^1} \int_K |\langle x, cu \rangle|^p  dx d\mu_{C,p}(c) =\int_{\unitsurf^1} \rho_{\IntBody_p K}(cu)^{-p} d\mu_{C,p}(c)
\end{align*}
 we arrive at the claim.
 
 In the second case, $-1< p<0$, Lemma~\ref{lem:2DimBodyEmbedLp}, applied again to $C^\circ \in \convexbodiesO(\CC)$, implies the existence of a measure $\nu_{C,p}$ on $\unitsurf^1$ such that
 \begin{align}\label{eq:prfPropRepCIntBodyNegPDarstC}
  \int_{\CC} \rho_{C^\circ}(z)^{-p} \phi(z) dz = \int_{\unitsurf^{1}} \left( \int_0^\infty t^{-p-1} \hat \phi(tc) dt \right) d\nu_{C,p}(c),
 \end{align}
 for every even Schwartz function $\phi$ on $\CC$. Note that $\nu_{C,p}$ can be chosen to be even. Since $\phi$ is even and (see, e.g., \cite{Koldobsky2005}*{Lem.~2.23})
 \begin{align}\label{eq:prfLpDarstFourier}
  \widehat{|t|^{-p-1}}(r) = 2 \Gamma(-p) \sin\left(\frac{\pi(p+1)}{2}\right) |r|^p, \quad r \in \RR,
 \end{align}
 we can rewrite the inner integral on the right-hand side to obtain
 \begin{align*}
  \int_0^\infty t^{-p-1} \hat \phi(tc) dt &= \frac{1}{2}\int_\RR |t|^{-p-1} \hat \phi(tc) dt \\
  &= \Gamma(-p) \sin\left(\frac{\pi(p+1)}{2}\right)\int_\RR |r|^{p} \tilde{\phi}_c (r) dr = c_p \int_0^\infty r^{p} \tilde{\phi}_c(r) dr,
 \end{align*}
 where we denote by $\tilde{\phi}_c$ the Fourier transform in $\RR$ of $t \mapsto \hat \phi(tc)$, and collect the constants into $c_p \in \RR$.
 
 If $f \in C^\infty(\CC^n)$ is even and has compact support, then the complex Radon transform $\mathcal{R}_u^\CC f$ is again even and smooth with compact support, and thus a Schwartz function. Taking now $\phi = \mathcal{R}_u^\CC f$, then $\tilde{\phi}_c = \mathcal{R}_{cu}^\RR f$, by Lemma~\ref{lem:ComplRealRadonFourier}, and, hence, 
 \begin{align*}
   \int_0^\infty t^{-p-1} \widehat{\mathcal{R}_u^\CC f}(tc) dt = c_p \int_0^\infty r^{p} \left(\mathcal{R}_{cu}^\RR f\right)(r) dr.
 \end{align*}
 Equation~\eqref{eq:prfPropRepCIntBodyNegPDarstC} therefore implies
 \begin{align*}
  \int_{\CC} \rho_{C^\circ}(z)^{-p} \left(\mathcal{R}_u^\CC f\right)(z) dz = c_p \int_{\unitsurf^{1}} \int_0^\infty r^{p} \left(\mathcal{R}_{cu}^\RR f\right)(r) dr \, d\nu_{C,p}(c),
 \end{align*}
 for every even $f \in C^\infty(\CC^n)$ with compact support. As $\rho_{C^\circ}$ and $\nu_{C,p}$ are even, this equation clearly also holds for functions $f$ that are not necessarily even. Moreover, by approximation, it holds for $f = \mathbbm{1}_K$, $K \in \starbodiesO(\CC^n)$, where $\mathcal{R}_u^\CC f = \CompParSec{K}{u}$ and $\mathcal{R}_{cu}^\RR f = \RealParSec{K}{cu}$ are the complex and real parallel section functions (see \eqref{eq:defCompParSec} and the comment below it). Consequently, for $u \in \unitsurf^{2n-1}$,
 \begin{align*}
  \rho_{\IntBody_{C,p} K}(u)^{-p} = \int_{\CC} \rho_{C^\circ}(z)^{-p} \CompParSec{K}{u}(z) dz = c_p \int_{\unitsurf^{1}} \int_0^\infty r^{p} \RealParSec{K}{cu}(r) dr \, d\nu_{C,p}(c),
 \end{align*}
 and, since $\nu_{C,p}$ is even, the right-hand side is equal to
 \begin{align*}
  \frac{c_p}{2} \int_{\unitsurf^{1}} \int_\RR |r|^{p} \RealParSec{K}{cu}(r) dr \, d\nu_{C,p}(c) = \frac{c_p}{2} \int_{\unitsurf^{1}} \rho_{\IntBody_p K}(cu)^{-p} \, d\nu_{C,p}(c),
 \end{align*}
 which yields the claim with $\mu_{C,p} = \frac{c_p}{2} \nu_{C,p}$.

 For $p=-1$, finally, we first show that \eqref{eq:prfPropRepCIntBodyNegPDarstC} holds for $d\nu_{C,-1} = \frac{1}{2\pi}\rho_{iC^\circ}(c)dc$. Indeed, let $\phi$ be an even Schwartz function on $\CC$. Using polar coordinates, the homogeneity of radial functions and the parity of $\phi$,
 \begin{align*}
  \int_\CC \rho_{C^\circ}(x) \phi(x) dx = \int_{\unitsurf^1}\frac{1}{2}\rho_{C^\circ}(c)  \int_\RR \phi(rc) dr dc,
 \end{align*}
 where the inner integral equals $\widehat{\phi(\cdot c)}(0)$. By \cite{Koldobsky2005}*{Lem.~2.11}, $\widehat{\mathcal{R}_c^\RR \hat\phi}(s) = (2\pi)^2 \phi(sc)$ for $s \in \RR$, that is,
 \begin{align*}
  \int_\RR \phi(rc) dr = \widehat{\phi(\cdot c)}(0) = \frac{1}{2\pi} \mathcal{R}_c^\RR \hat\phi(0) = \frac{1}{2\pi} \int_{\langle x, c\rangle = 0} \hat \phi(x) dx = \frac{1}{2\pi}  \int_\RR \hat\phi(itc) dt.
 \end{align*}
 Changing the outer integration and by $\rho_{C^\circ}(-ic) = \rho_{iC^\circ}(c)$ and the parity of $\hat \phi$, we arrive at \eqref{eq:prfPropRepCIntBodyNegPDarstC} for $p=-1$ and $d\nu_{C,-1} = \frac{1}{2\pi} \rho_{iC^\circ}(c) dc$.

 Next, we repeat the steps from the previous part ($-1<p<0$) to obtain
 \begin{align*}
  \int_{\unitsurf^{1}} \left( \int_0^\infty t^{-p-1} \widehat{\mathcal{R}_u^\CC f}(tc) dt \right) d\nu_{C,-1}(c) = c_p \int_{\unitsurf^{1}} \int_0^\infty r^p \left(\mathcal{R}_{cu}^\RR f\right)(r) dr \, d\nu_{C,-1}(c)
 \end{align*}
 for every even $f \in C^\infty(\CC^n)$ with compact support. Recalling the convergence~\eqref{eq:defDistrtqplusConvMin1} of the family of distributions $r^p_+$ from~\eqref{eq:defDistrtqplus} as $p\to -1^+$, we deduce by dominated convergence (as $\widehat{\mathcal{R}_u^\CC f}$ is a Schwartz function and $\mathcal{R}_{cu}^\RR f$ has support uniformly bounded in $c$) that
 \begin{align}\label{eq:prfPropRepCIntBodyPMin1AfterLim}
  \int_{\unitsurf^{1}} \left( \int_0^\infty \widehat{\mathcal{R}_u^\CC f}(tc) dt \right) d\nu_{C,-1}(c) = \pi \int_{\unitsurf^{1}} \left(\mathcal{R}_{cu}^\RR f\right)(0) \, d\nu_{C,-1}(c),
 \end{align}
 where we also used that $c_p \Gamma(p+1) \to \pi$. Combining \eqref{eq:prfPropRepCIntBodyPMin1AfterLim} with \eqref{eq:prfPropRepCIntBodyNegPDarstC},
 \begin{align*}
  \int_\CC \rho_{C^\circ}(x) \left(\mathcal{R}_u^\CC f\right)(x) dx = \pi \int_{\unitsurf^{1}} \left(\mathcal{R}_{cu}^\RR f\right)(0) \, d\nu_{C,-1}(c),
 \end{align*}
 which, since $\nu_{C,-1}$ is even and by approximating $\mathbbm{1}_K$ by smooth functions $f$, implies
 \begin{align*}
  \rho_{\IntBody_{C,-1}K}(u) = \int_\CC \rho_{C^\circ}(z) \CompParSec{K}{u}(z) dz = \pi \int_{\unitsurf^{1}} \!\! \RealParSec{K}{cu}(0) \, d\nu_{C,-1}(c) = \frac{1}{2}\int_{\unitsurf^1}\!\! \rho_{\IntBody K}(cu) \rho_{iC^\circ}(c) dc
 \end{align*}
 for every $K \in \starbodiesO(\CC^n)$ and $u \in \unitsurf^{2n-1}$, yielding the claim.
\end{proof}

Note that we needed to consider the case $p=-1$ separately in the proof as we can not apply \eqref{eq:prfLpDarstFourier} for $p=-1$.

%
%
%
%

\bigskip

\noindent The proof Theorem~\ref{prop:compLp=realLp} is now a direct consequence of Proposition~\ref{prop:repCIntBodyByLpIntBody}.
\begin{proof}[Proof of Theorem~\ref{prop:compLp=realLp}]
 Let $p>-1$ and $C \in \convexbodiesO(\CC)$. If $K \in \starbodiesO(\CC^n)$ is $\unitsurf^1$-invariant, then $\IntBody_p K$ is $\unitsurf^1$-invariant as well, by $\SL(2n,\RR)$-contravariance. Consequently, $\rho_{\IntBody_p K}(cu) = \rho_{\IntBody_p K}(u)$ for all $u \in \unitsurf^{2n-1}$ and $c \in \unitsurf^1$, and by Proposition~\ref{prop:repCIntBodyByLpIntBody},
 \begin{align*}
  \rho_{\IntBody_{C,p} K}(u)^{-p} = \int_{\unitsurf^1} \rho_{\IntBody_p K}(cu)^{-p} d\mu_{C,p}(c) = \mu_{C,p}(\unitsurf^1)\rho_{\IntBody_p K}(u)^{-p}, \quad u \in \unitsurf^{2n-1},
 \end{align*}
 which yields the claim for $d_{C,p} = \mu_{C,p}(\unitsurf^1)^{-1/p}$ (noting that $\mu_{C,p}(\unitsurf^1) > 0$ as it is equal to the radius of $\IntBody_{C,p} B$ for some suitably chosen ball $B \subseteq \CC^n$). 
\end{proof}

\bigskip

\noindent
We continue by proving Theorem~\ref{mthm:ineqComplexLpVsLp}, which we can now state with the (technical) equality conditions, depending on the measure $\mu_{C,p}$ from Proposition~\ref{prop:repCIntBodyByLpIntBody}.

\begin{theorem}\label{thm:ineqComplexLpVsLp}
	Suppose that $C \in \convexbodiesO(\CC)$ is origin-symmetric and $-1\leq p< 1$ is non-zero. If $K \in \starbodiesO(\CC^n)$, then
	\begin{align}\label{eq:thmIneqComplLpVsLpClaim}
		\frac{\Vol{2n}{\IntBody_{C,p}K}}{\Vol{2n}{\IntBody_{C,p}B^{2n}}} \leq \frac{\Vol{2n}{\IntBody_p K}}{\Vol{2n}{\IntBody_{p}B^{2n}}}.
	\end{align}
	If $\mu_{C,p}$ has infinite support equality holds if and only if $\IntBody_p K$ is $\unitsurf^1$-invariant, and for all other $\mu_{C,p}$ equality holds if and only if $ \IntBody_p K = c_1 \overline{c_2} \IntBody_p K$ whenever $c_1,c_2 \in \supp \mu_{C,p}$.
\end{theorem}
\begin{proof}
By Proposition~\ref{prop:repCIntBodyByLpIntBody},
\begin{align*}
 \Vol{2n}{\IntBody_{C,p}K} = \frac{1}{2n} \int_{\unitsurf^{2n-1}} \left(\int_{\unitsurf^1} \rho_{\IntBody_p K}(cu)^{-p} d\mu_{C,p}(c) \right)^{-2n/p} du.
\end{align*}
We apply Jensen's inequality to the inner integral to obtain
\begin{align}\label{eq:prfThmIneqComplexLpVsLpJensen}
 \Vol{2n}{\IntBody_{C,p}K} \leq \frac{\mu_{C,p}(\unitsurf^1)^{-1-2n/p}}{2n} \int_{\unitsurf^{2n-1}} \int_{\unitsurf^1} \rho_{\IntBody_p K}(cu)^{2n} d\mu_{C,p}(c) du.
\end{align}
Changing the order of integration using Fubini's theorem,
\begin{align*}
 \Vol{2n}{\IntBody_{C,p}K} \leq \mu_{C,p}(\unitsurf^1)^{-1-2n/p} \int_{\unitsurf^1} \Vol{2n}{\overline{c}\IntBody_p K} d\mu_{C,p}(c) = \mu_{C,p}(\unitsurf^1)^{-2n/p}\Vol{2n}{\IntBody_p K},
\end{align*}
we arrive at the desired inequality, since $\rho_{\IntBody_{C,p} B^{2n}} = \mu_{C,p}(\unitsurf^1)^{-1/p} \rho_{\IntBody_p B^{2n}}$ by \eqref{eq:proprepCIntBodyByLpIntBody}.

\medskip

Equality holds in \eqref{eq:thmIneqComplLpVsLpClaim} if and only if it holds in \eqref{eq:prfThmIneqComplexLpVsLpJensen}, that is, by the equality conditions of Jensen's inequality, exactly if for almost every $u \in \unitsurf^{2n-1}$ there exists $d_u \in \RR$ such that
\begin{align*}
 \rho_{\IntBody_p K}(cu)^{-p} = d_u, \quad \text{ for $\mu_{C,p}$-a.e. } c \in \unitsurf^1.
\end{align*}
Note that, by the continuity of $\rho_{\IntBody_p K}$, this holds indeed for \emph{all} $u \in \unitsurf^{2n-1}$. Consequently, $\overline{c_1} \IntBody_p K = \overline{c_2} \IntBody_p K$ for $\mu_{C,p}$-almost all $c_1,c_2 \in \unitsurf^1$, that is, again by continuity,
\begin{align}\label{eq:prfThmIneqComplexLpVsLpJensenEqCas}
 \IntBody_p K = c_1 \overline{c_2} \IntBody_p K, \quad c_1, c_2 \in \supp \mu_{C,p},
\end{align}
which yields the equality condition for $\mu_{C,p}$ with finite support. If $\supp \mu_{C,p}$ is infinite, then, by compactness of $\unitsurf^1$, for every $\varepsilon > 0$, there exist $c_1, c_2 \in \supp \mu_{C,p}$ such that $|c_1\overline{c_2} - 1| < \varepsilon$. Iterating \eqref{eq:prfThmIneqComplexLpVsLpJensenEqCas}, we obtain that the map $c \mapsto \rho_{\IntBody_p K}(cu)$ is constant on the set $\{(\overline{c_1}c_2)^k: k \in \NN\}\subseteq \unitsurf^1$, which is $\varepsilon$-close to every $c \in \unitsurf^1$. Since $\varepsilon>0$ was arbitrary and $\rho_{\IntBody_p K}$ is continuous, we conclude that $\IntBody_p K$ must be $\unitsurf^1$-invariant. The converse follows easily from \eqref{eq:proprepCIntBodyByLpIntBody}.
\end{proof}

\bigskip

\begin{proof}[Proof of Corollary~\ref{mcor:CompIntIneq}]
 Suppose that $C \in \convexbodiesO(\CC)$ is origin-symmetric and let $K \in \starbodiesO(\CC^n)$. Then, by inequality~\eqref{eq:thmIneqComplLpVsLpClaim} and Theorem~\ref{thm:lpIntIneqAdamczakEtAl} (resp. Busemann's intersection inequality~\eqref{eq:BusIntersectIneq} for $p=-1$) it follows that
 \begin{align}\label{eq:prfThmBIneq}
  \frac{\Vol{2n}{\IntBody_{C,p}K}}{\Vol{2n}{\IntBody_{C,p}B^{2n}}} \leq \frac{\Vol{2n}{\IntBody_p K}}{\Vol{2n}{\IntBody_p B^{2n}}} \leq \frac{\Vol{2n}{K}^{2n+p}}{\Vol{2n}{B^{2n}}^{2n+p}}.
 \end{align}
 Equality holds for $p=-1$ and $K \in \starbodiesO(\CC^n)$ if and only if there is equality in \eqref{eq:thmIneqComplLpVsLpClaim} and \eqref{eq:BusIntersectIneq}. Since $d\mu_{C,-1} = \frac{1}{2}\rho_{iC^\circ}(c) dc$ has infinite support (equal to $\unitsurf^1$), the equality cases of Theorem~\ref{thm:ineqComplexLpVsLp} imply that $\IntBody K$ must be $\unitsurf^1$-invariant, whereas the equality cases of \eqref{eq:BusIntersectIneq} imply that $K$ must be an origin-symmetric ellipsoid. Consequently, as the intersection body map is injective on origin-symmetric star bodies (see, e.g., \cite{Gardner2006}*{Thm.~8.1.3}), we conclude that $K$ is an $\unitsurf^1$-invariant ellipsoid, which is equivalent to $K$ being an origin-symmetric Hermitian ellipsoid.
 
 If, on the other hand $K$ is an origin-symmetric Hermitian ellipsoid, then there clearly is equality in \eqref{eq:prfThmBIneq}.
\end{proof}

\bigskip

\subsection*{Acknowledgements}
The authors would like to thank Boris Rubin, Franz Schuster and Thomas Wannerer for their valuable comments and suggestions.

\bibliographystyle{abbrv}
\bibliography{ComplLpIntersBodies}

\end{document}